\documentclass[letter,11pt]{article}%
\usepackage{amsmath}
\usepackage{amsfonts}
\usepackage{amssymb}
\usepackage{authblk}
\usepackage{bm}
\usepackage{bbm}
\usepackage{booktabs}
\usepackage{caption}
\usepackage{float}
\usepackage{relsize}
\usepackage{subfig}
\usepackage{enumitem}
\usepackage{tabu}
\usepackage{tabularx,ragged2e,booktabs,caption}
\usepackage{pbox}
\usepackage[abs]{overpic}
\usepackage{graphicx}
\usepackage[usenames,dvipsnames]{color}
\usepackage{rotating}
\usepackage[height=9in,left=1in,right=0.75in,bottom=1in]{geometry}
\usepackage[breaklinks=true,bookmarksopen=true,colorlinks=true,citecolor=blue]%
{hyperref}
\RequirePackage[longnamesfirst]{natbib}
\usepackage[left]{lineno}
\setlist[itemize]{leftmargin=*}
\providecommand{\U}[1]{\protect\rule{.1in}{.1in}}
\newtheorem{theorem}{Theorem}[section]

\newtheorem{lemma}[theorem]{Lemma}

\newtheorem{proposition}[theorem]{Proposition}
\newtheorem{remark}{Remark}[section]

\newenvironment{proof}[1][Proof]{\noindent\textbf{#1.} }{\ \rule{.52em}{.52em}}

\makeatletter

\@addtoreset{lemma}{Appendix B} \makeatother

\DeclareMathOperator*{\argmin}{argmin}
\newcommand{\xz}{x_0}
\newcommand{\intgx} {\Upsilon_{x_0}}
\newcommand{\intgxo} {\Upsilon_{1,x_0} }
\newcommand{\intgxt} {\Upsilon_{2,x_0} }
\newcommand{\doublesum}{\sum_{j_1=1}^{J_1}\sum_{j_2=1}^{J_2}}
\newcommand{\bjqr} { b ^{( r)}_{J_1,J_2}}
\newcommand{\bjr} { b ^{( r)}_{J_1,J_2}}
\newcommand{\ftz} {f^{(2,0)}}
\newcommand{\foo} {f^{(1,1)}}
\newcommand{\fzt} {f^{(0,2)}}
\hypersetup{pdftitle={}, pdfsubject={}, pdfauthor={}, pdfkeywords={} }

\begin{document}

\title{Posterior Contraction and Credible Sets for Filaments of Regression Functions}
\author[1]{Wei Li \thanks{Department of Mathematics, 215 Carnegie Building, Syracuse, New
		York 13244, USA. E-mail: \href{mailto:wli169@syr.edu}{wli169@syr.edu}  }}
\author[2]{Subhashis Ghosal\thanks{Department of Statistics, North Carolina State University, 
		Raleigh, North Carolina 27695, USA. E-mail: \href{mailto:sghosal@stat.ncsu.edu}{sghosal@stat.ncsu.edu. }   }}
\affil[1]{Department of Mathematics, Syracuse University}
\affil[2]{Department of Statistics, North Carolina State University}
\date{March 24, 2020}
\maketitle


\begin{abstract}
 A filament consists of local maximizers of a smooth function $f$ when moving in a certain direction. A filamentary structure is an important feature of the shape of an object and is also considered as an important lower dimensional characterization of multivariate data. There have been some recent theoretical studies of filaments in the nonparametric kernel density estimation context. This paper supplements the current literature in two ways. First, we provide a Bayesian approach to the filament estimation in regression context and study the posterior contraction rates using a finite random series of B-splines basis. Compared with the kernel-estimation method, this has a theoretical advantage as the bias can be better controlled when the function is smoother, which allows obtaining better rates. Assuming that $f: \mathbb{R}^2 \mapsto \mathbb{R}$ belongs to an isotropic H\"{o}lder class of order $\alpha \geq 4$, with the optimal choice of smoothing parameters, the posterior contraction rates for the filament points on some appropriately defined integral curves and for the Hausdorff distance of the filament are both $(n/\log n)^{(2-\alpha)/(2(1+\alpha))}$. Secondly, we provide a way to construct a credible set with sufficient frequentist coverage for the filaments. We demonstrate the success of our proposed method in simulations and one application to earthquake data.

	\vspace{2mm}
	
\end{abstract}

\section{Introduction}
\noindent
There is a large body of literature on the problem of estimating intrinsic lower dimensional structure of multivariate data. A filament or a ridgeline is one of such geometric objects that draws a lot of attention in the recent years. Intuitively speaking, a filament consists of local maximizers of a smooth function $f$ (say a density or a regression function) when moving in a certain direction. Roughly speaking these are generalized modes that reside on hyperplanes that are normal to the steepest ascent direction.

The filamentary structure (ridges) together with the valleys (i.e. local minimizers counterparts of ridges), critical points are the main features of the shapes of objects. They are common in medical images, satellite images and many three dimensional objects. One important example comes from the study of the cosmic web --- a large scale web structure of galaxies (clusters) connected by long threads composed of sparse hydrogen gas. These intergalactic connections are believed to trace the filaments of dark matters. The discovery and study of the dark matter is a key challenge of cosmology. For more references, see \citet{novikov2006skeleton}, \citet{dietrich2012filament} and \citet{chen2015astronomy}.

The filament estimation falls into a broad category of data analytic methods that are called topological data analysis which is used to find intrinsic structure in data \citep{wasserman2016}. In particular, it is closely related to manifold learning problem. Manifold learning problem assumes that the data points are generated from some a priori unknown lower dimensional structure with background noises. \citet{arias2006annals} developed a test to detect if a dataset contains some small fraction of data points that are supported on a curve. \citet{genovese2012manifold} and \citet{genovese2014annals} studied the problem of estimation of the manifold. They showed that the ridge of a density function can serve as a surrogate to the manifold and can be estimated with a better rate. Since the filaments can be considered as generalized modes,  relevant literature includes those on mode (and maximum) estimation, for instance, \citet{shoung2001least, facer2003nonparametric} and a recent paper by \citet{yoo2019mode} in the Bayesian framework. 

The statistical properties of the estimated filaments, like convergence rates and limiting distribution, have been studied in a few recent papers. \citet{genovese2014annals}  established the convergence rates of the filament obtained from the kernel density estimation. \citet{chen2015annals} provided Berry-Esseen type results for the limiting distribution. With a different approach, \citet{qiao2016annals} also established the convergence rates and the extreme value type results for limiting distribution. Constructing confidence bands in nonparametric problems is also a very important topic in itself. Some recent papers include \citet{claeskens2003bootstrap}, \citet{gine2010confidence} and \citet{chernozhukov2013gaussian,chernozhukov2014anti}. But these works focus primarily on the standard functions, e.g., regressions and density functions. It is worth to point out that in the problems of filaments, the estimation and inference is performed for a set, or more broadly speaking for some geometric object \citep{molchanov2006theory}. A similar example is the level set of a certain function \citet{jankowski2012confidence, mason2009asymptotic, mammen2013confidence}. Some other examples arise in econometrics literature where partially identified parameters (as a set) is the object of interest; see \citet{chernozhukov2007estimation, chernozhukov2015inference} and references therein. Very little is known about how to make inference about a geometric target in general and this remains a very important area of study. For the problems of filaments, \citet{chen2015annals} developed a  bootstrap-based method for uncertainty quantification. 

Frequentist properties of Bayesian procedure for an unknown function has also been an area receiving much attention. In particular, posterior contraction rates for functions in different $L_r$-metrics have been studied in  \citet{gine2011rates, castillo2014bayesian} and \citet{yoo2016supremum}. There also have been many studies on constructing credible regions which have the right frequentist coverage. \citet{szabo2015frequentist} studied adaptive $L_2-$credible regions in Gaussian white noise model. \citet{yoo2016supremum} constructed $L_{\infty}$-credible regions in the multivariate nonparametric regression setting with known smoothness condition. \citet{knapik2011bayesian} studied the frequentist coverage of credible sets in nonparametric inverse problems. 
\citet{belitser2017coverage} studied credible sets in mildly ill-posed inverse signal-in-white-noise model.  \citet{belitser2015needles} studied uncertainty quantification for the unknown, possibly sparse, signal in general signal with noise models. 
\citet{van2017uncertainty} studied credible sets using the horseshoe prior in the sparse multivariate normal means model  in an adaptive setting. 
\citet{ray2017adaptive} studied Bernstein--von Mises theorems for adaptive nonparametric Bayesian procedures in the Gaussian white noise model. \citet{yoo2019mode} studied Bayesian mode and maximum estimation and provided credible sets with good coverage. \citet{belitser2018empirical} studied uncertainty quantification for high dimensional linear
regression models and their results are also extended to high dimensional additive nonparametric regression models. Using credible regions with sufficiently frequentist coverage, one can obtain confidence regions for the truth in the frequentist sense relatively easily from the posterior distribution. This is especially appealing when the object to be studied is complicated.

So far in the literature, the study of filaments is limited only to densities only using the kernel approach. This paper supplements the current literature in two ways. First, we provide a Bayesian approach to the filament estimation in regression context and study the contraction rates using a finite random series of B-splines basis. This has theoretical advantages as the bias can be better controlled when the function is more than ``minimally smooth" (i.e. more than four times differentiable). Specifically, better rates are obtained for the filament points on some appropriately defined integral curves and for the Hausdorff distance of the filament which are both $(n/\log n)^{(2-\alpha)/(2(1+\alpha))}$ (see Remarks  \ref{rates:filaments_on_curves_remark} and \ref{rates:hausdorff_remark}). Secondly, we provide a way to construct credible set with sufficient frequentist coverage for the filaments (see Theorem  \ref{credible_sets}). In contrast to the bootstrap-based confidence region proposed by \citet{chen2015annals} which gives a band-shape region, our valid credible region consists of filaments from posterior samples. Another difference is that the inferential target in this paper is the true quantity itself, while in \citet{chen2015annals} the inference is directed towards to the debiased quantity. Although relying on some existing results in the literature, these additional results we obtain offer a new perspective how a Bayesian approach can be applied to a complicated nonparametric problem in a rather intuitive way.

Before we move on to the formal definition of the filament, it is worth to point out that some other possible definitions of filaments have also been discussed and studied in mathematics and computer sciences literature; see \citet{emberly} for more details. In this paper, we study the filament as introduced in \citet{chen2015annals} and \citet{qiao2016annals}.

This paper is organized as follows. Notation and background materials are given in Section 2. The model is formally introduced in Section 3 along with the descriptions of the prior distribution and the posterior distribution. Technical assumptions are given in Section 4. The main results in posterior contraction and credible region are presented in Section 5. Simulation results and an application to an earthquake data are presented in Section 6 and 7 respectively. All proofs are given in Section 8.

\section{Notations and preliminaries }

Let $\mathbb{N}=\{1,2,..,\}$, $\mathbb{N}_0=\{0,1,2,3,\ldots\}$. Given two real sequences $a_n$ and $b_n$, $a_n=O(b_n)$ or $a_n \lesssim b_n$ means that $a_n/b_n$ is bounded, while $a_n=o(b_n)$ or $a_n \ll b_n$ means that both $a_n/b_n \to 0.$ Also $a_n \asymp b_n$ means that $a_n=O(b_n)$ and $b_n=O(a_n)$. For a sequence of random elements $Z_n$, $Z_n=O_{p}(a_n)$ means that $\mathrm{P}(|Z_n|\leq C a_n) \to 1$ for some constant $C>0$.

For a vector $x \in \mathbb{R}^d$, we define $\|x\|_p=(\sum_{k=1}^d|x_k|^p)^{1/p}$ for $0 \leq p< \infty$, $\| x \|_\infty=\max_{1\leq k \leq d }|x_k|$ and write $\| x \|$ for $\| x\|_2$. The scalar product of two vectors $x$ and $y$ will be written as $x^Ty$ or $\langle x, y  \rangle$. For an $m \times m$ matrix $A$, let $\|A \|_{(2,2)}=(\lambda_{\max} (A^T A))^{1/2}$, where $\lambda_{\max}$ denotes the largest eigenvalue; $\| A \|_{(\infty,\infty)}=\max_{1\leq i \leq m}\sum_{j=1}^m|a_{ij}|$ and $\|A\|_F=\sqrt{\text{tr}(A^TA)}$ the Frobenius norm of matrix $A$. Given another $m\times m$ matrix $B$, $B\leq A$ means $A-B$ is non-negative definite.
We also denote an $n$ by $n$ identity matrix by $I_{n}$.

For $f: U\mapsto \mathbb{R} $ on some bounded set $U \subset \mathbb{R}^d  $. Let $\|f\|_p$ be the $L_p$ norm and $\|f \|_{\infty}=\sup_{x \in U} |f( x)|$. For $g: U \mapsto \mathbb{R} $ on some bounded set $U \subset \mathbb{R}^d  $, let $\nabla g$ be the gradient of $g$, which is a $d \times 1$ vector of functions. For a $d$-dimensional multindex $r=(r_1,\ldots,r_d) \in \mathbb{N}_0^d$, let $D^{r}$ be the partial derivative operator $\partial^{|r|} / \partial x_1^{r_1} \cdot\cdot\cdot \partial x_d^{r_d}$ where $|r|=\sum_{k=1}^dr_k$.

The H\"{o}lder Space $\mathcal{H}^{\alpha}([0,1]^d)$ of order $\alpha>0$ consists of functions $f: [0,1]^d \mapsto \mathbb{R}$ such that $\|f \|_{\alpha,\infty} < \infty$, where $\| \cdot \|_{\alpha,\infty}$ is the H\"{o}lder norm
\[
\|f \|_{\alpha,\infty} = \max_{r: |r|\leq \lfloor \alpha \rfloor}\sup_{ x}|D^{ r }f ( x )|+ \max_{ r: | r|= \lfloor \alpha \rfloor}\sup_{x, y:x \neq y}\frac{|D^{ r}f( x ) - D^{ r}f( y) |}{\|x -y \|^{\alpha- \lfloor \alpha \rfloor }},
\]
where $\sup$ is taken over the support of $f$ and $ \lfloor \alpha \rfloor$ is the largest integer strictly smaller than $\alpha$.

The filament or the ridge line of a smooth function defined on $\mathbb{R}^2$ is a collection of points at which the gradient of the function is orthogonal to the eigenvector of its Hessian that corresponds to the most negative eigenvalue (if exists).
The filament point (the point on the filament) is a generalization of mode of the function. To see this connection, recall a well-known result that tests for local maximum point (mode).

Let $f: \mathbb{R}^2 \mapsto \mathbb{R}$ be a smooth function, $\nabla f= (f^{(1,0)}, f^{(0,1)})^T$ be the gradient and $Hf$ be the Hessian. Recall a test for a local maximum point is the following
\[
a^T\nabla f(x )=0, \quad \quad a^T Hf( x)a<0,
\]
for all nonzero vector $a$.
Let $V(x)$ be the eigenvector of $Hf(x)$ that corresponds to the smallest eigenvalues $\lambda(x)$. A point $x \in \mathbb{R}^2$ is called a filament point or ridge point if
\begin{equation}
V^T(x)\nabla f(x)=0, \quad \quad V^T(x)Hf(x)V(x)<0.
\end{equation}
Therefore, a filament point is a point at which the function has a local maximum along the direction given by $V$. Note also that $V^T(x)Hf(x)V(x)<0$ is equivalent to $\lambda(x)<0$.

More generally, for $f$ defined on $\mathbb{R}^d$, $0\leq s \leq d-1$, the eigenvectors of $Hf(x)$ can be used to define two orthogonal spaces, namely, a $(d-s)$-dimensional normal space (corresponding to $(d-s)$-eigenvectors with smallest eigenvalues) and an $s$-dimensional tangent space (corresponding to the remaining eigenvectors). An $s$-dimensional filament point on $\mathbb{R}^d$ is a point where the gradient of $f$ is orthogonal to the normal space and the eigenvalues associated with the normal space are all (strongly) negative. Alternatively, such a point $x$ can be regarded as a point where $f$ attains the local maximum in the affine space spanned by the normal space translated by $x$. The modes are then simply 0-dimensional filaments. The 1-dimensional filament on $\mathbb{R}^2$ is of primary interest in our discussion here.

 From now on, $f$ is assumed to be some smooth bivariate regression function. Suppose that the Hessian matrix $Hf(x)$ of $f$ at $x$ has eigvenvector $V(x)$ corresponding to the smaller eigenvalue $\lambda(x)$. The filament $\mathcal{L}(f)$ of the regression function $f$ is  formally defined as
\begin{equation}
 \mathcal{L}=\mathcal{L}(f):= \{x: \langle \nabla f(x), V(x)  \rangle=0 \text{ and } \lambda(x) < 0 \}.
\end{equation}

We also introduce an integral curve, which is the solution to the following differential equation
\begin{equation}
\frac{d \intgx(t) }{dt}=V(\intgx(t)), \quad \intgx(0)=\xz,
\end{equation}
where $x_0$ is some starting point from a sufficiently rich set $\mathcal{G}$ to be described in Section 4. We define the ``hitting time" of the filament by traversing the integral curve starting at a point $x_0$,
\begin{equation}
t_{\xz}=\argmin_t \{ |t|\geq 0:  \langle \nabla f(\intgx(t)), V(\intgx(t))  \rangle=0, \lambda(\intgx(t))< 0 \}.
\end{equation}
The integral curves will be our intermediate object for the study of the filament (as a collection of points on these curves). It is obvious that $\intgx(t_{\xz})\in \mathcal{L}$. Taking a plug-in estimation's point of view, $\intgx$ and $t_{\xz}$ can be estimated, and therefore a filament point (on certain integral curve) can be estimated. Through the pointwise comparison between the estimated filament point and the true filament point over a large collection of starting points, one can assess the performance of the estimation procedure. This idea is put forward in \cite{qiao2016annals} and is useful for our study.

\section{Model, prior and posterior}

Throughout the paper, let $d=2$, thus $x=(x_1,x_2)$.  We consider the nonparametric regression model,
\[
Y_i=f(X_i)+\varepsilon_i,
\]
where $\varepsilon_i$ are independent and identically distributed (i.i.d) Gaussian random variables with mean $0$ and variance $\sigma^2$ for $i=1,\ldots,n$.
Without loss of generality, we let $X_i$ takes values in $[0,1]^2$.  Let $Y=(Y_1,
\ldots,Y_n)^T$, $X=(X_1^T,\ldots,X_n^T)^T$, $F=(f(X_1),\ldots,f(X_n))^T$ and $\varepsilon=(\varepsilon_1,\ldots,\varepsilon_n)^T$, then we can write $ Y=F+\varepsilon$. 

To model $f$, we shall use B-spline function as our basis function. Given some $q\in \mathbb{N}$, $N \in \mathbb{N} $, for a sequence of  knots $0= t_{-(q-1)}= \cdots=t_0<\cdots<t_{N+1}=\cdots =t_{N+q}=1$, denote the univariate B-spline function of order $q$ by $B_{i,q}(x)=(t_i-t_{i-q})[t_{i-q},\cdots,t_i](\cdot - x)_{+}^{q-1}, i=1,2,\ldots,N+q$ \citep{de1978practical}. Here $[t_{i-q},\cdots,t_i](\cdot - x)_{+}^{q-1}$ is the divided difference of the function $y \mapsto (y - x)_{+}^{q-1}$, where $(y - x)_{+}^{0}=\mathbbm{1}(y\geq x)$, $(y - x)_{+}^{q-1}=(y-x)^{q-1}\mathbbm{1}(y\geq x)$. By construction, $B_{i,q}(x)>0$ on $(t_{i-q},t_i)$ and $\sum_{i=1}^{N+q}B_{i,q}(x)=1 $. Now to construct a basis on $\mathbb{R}^2$ for $x=(x_1,x_2) \in \mathbb{R}^2$, define
$b_{J_1, J_2, q_1, q_2}(x)=(B_{j_1,q_1}(x_1)B_{j_2, q_2}(x_2): 1\leq j_k \leq J_k, k=1,2 )^T$ to be a vector of tensor product of B-splines functions, with possibly  different orders $q_k$ and knot sequences in different directions, i.e, $0=t_{k, -(q_k-1)}=\cdots=t_{k,0}<t_{k,1}<\ldots<t_{k,N_k}<t_{k,N_k+1}=\cdots=t_{k,N_{k}+q_k}=1$ for $k=1,2$. Here $N_k$ denotes the number of interior points and $J_k=q_k+N_k$ denotes the number of basis functions on the k-th coordinate. The elements of this vector is assumed to be in dictionary order according to their indices.  For each $k=1,2$, define $\delta_{k,\ell}=t_{k,\ell+1}-t_{k,\ell}$ for $\ell =0,\ldots,N_k$ and assume that $\max_{1\leq \ell \leq N_k}\delta_{k,\ell}/\min_{1\leq \ell \leq N_k } \delta_{k,\ell} \leq C$ for some $C>0$. This assumption is clearly satisfied for the uniform partition. Whenever $q_1, q_2$ are considered fixed, we shall suppress the subscripts $q_1, q_2$ in our notations of B-spline functions, for instance we write $b_{J_1,J_2}$ for $b_{J_1, J_2, q_1, q_2}$. Following a similar set-up in \citet{yoo2016supremum}, we first put a random tensor-product B series prior on $f$.  Let $q_1, q_2$ fixed and $N_k=N_k(n)$ and hence $J_k=J_k(n)$. Letting $f(X_i)=b^T_{J_1, J_2}(X_i)\theta$ for some vector $\theta \in \mathbb{R}^{J_1J_2}$,
our model becomes
\[
Y|( X, \theta,\sigma^2) \sim {\rm N} ( B\theta,\sigma^2 I_n ),
\]
where $B=(b_{J_1,J_2}(X_1),\ldots, b_{J_1,J_2}(X_n)  )^T$.

Throughout we will use superscript * to denote the true values. Even though our model for $Y$ is Gaussian, the true error terms are only required to be sub-Gaussian. Specifically, we assume the data are i.i.d from some true distribution ${\rm P}_0$ where $Y_i=f^*(X_i)+\varepsilon_i$ with i.i.d. sub-Gaussian $\varepsilon_i$ whose mean is $0$ and variance $\sigma^2_0$ for $i=1,\ldots,n$. Our study allows both fixed and random design cases. If $\{ X_i: i=1,\ldots,n\}$ are considered fixed data points, we assume that for some cumulative distribution function $\mathbb{G}$ with positive and continuous density
\begin{equation}
\sup_{x \in [0,1]^2}|\mathbb{G}_n(x) -\mathbb{G}(x)|=o\Bigl( \frac{1}{J_1J_2}  \Bigr)  \label{empirical:eqn},
\end{equation}
where $\mathbb{G}_n$ is the empirical distribution of $\{X_i: i=1,\ldots,n \}$. If $X_i \overset{{ \rm i.i.d. }}{\sim} \mathbb{G}$, then the above condition is satisfied with probability tending to one as long as $J_1 \asymp J_2 \asymp o(n^{1/4})$ by Donsker's theorem. In both cases, we shall use $\mathbb{D}_n$ to denote all observations.

We assign $\theta|\sigma^2 \sim {\rm N}(\theta_0,\sigma^2 \Lambda_0  )$ for some $\theta_0 \in \mathbb{R}^{J_1J_2}$, assuming that for some constants $0< c_1\leq c_2 <\infty$ it holds that
\[
c_1I_{J_1 J_2} \leq \Lambda_0 \leq c_2I_{J_1J_2},
\]
where $I_{J_1J_2} $ is a $J_1J_2$ by $J_1J_2$ identity matrix.
It follows then
\begin{align*}
\theta|\mathbb{D}_n, \sigma^2 &\sim \mathrm{N} \left((\Lambda_0^{-1}+B^TB)^{-1} (B^TY+\Lambda_0^{-1}\theta_0), \sigma^2( B^TB+\Lambda_0^{-1})^{-1} \right).
\end{align*}

The posterior distribution for $f(x)$ and its partial derivatives are then obtained accordingly. Define the vector
\[b^{( r)}_{J_1,J_2}(x)=\Bigg( \frac{\partial^{r_1}}{\partial x_1^{r_1}}B_{j_1,q_1}(x_1)\frac{\partial^{r_2}}{\partial x_2^{r_2}}B_{j_2,q_2}(x_2): 1\leq j_k\leq J_k, k=1,2   \Bigg)^T .\]

Therefore,
\begin{align*}
&\Pi(D^{r}f |\mathbb{D}_n,\sigma^2) \sim \text{GP}(A_{r}Y+C_{r} \theta_0, \sigma^2 \Sigma_{r}), \\
&A_{ r}(x)={\bjr(x)}^T (B^TB+\Lambda_0^{-1} )^{-1} B^T,\\
&C_{ r}(x)={\bjr(x)}^T(B^TB+\Lambda_0^{-1} )^{-1} \Lambda_0^{-1},  \\
&\Sigma_{ r}(x,y)={\bjr(x)}^T (B^TB+\Lambda_0^{-1} )^{-1}{\bjr(y)},
\end{align*}
where GP denotes a Gaussian process. Notice that under $\mathrm{P}_0$, $A_{ r}(x)  \varepsilon/ \sigma_0$ is a mean-zero process with a sub-Gaussian tail.

To handle $\sigma^2$, we can either put a conjugate inverse-gamma prior $\sigma^2 \sim \text{IG}(a/2,b/2)$ with shape parameter $a/2>2$ and rate parameter $b/2>0$ or plug-in an estimate for $\sigma^2$. Since the theory does not make much difference, for ease of exposition, we shall use the second approach, to be called the empirical Bayes method.
The empirical Bayes has the following posterior distribution \citep{yoo2016supremum}
\begin{equation}
\Pi(D^{r}f |\mathbb{D}_n) \sim \text{GP}(A_{r}Y+C_{r} \theta_0, \hat{\sigma}^2 \Sigma_{r}) \label{posterior:eqn},
\end{equation}
where
\begin{equation}
\hat{\sigma}^2=n^{-1}(Y-B\theta_0)^T(B \Lambda_0 B^T+I_n)^{-1}(Y-B\theta_0).
\end{equation}

\section{Assumptions}

We follow the standard assumptions in \citet{qiao2016annals}. For convenience, we let $d^2f(x)=(\ftz(x), \foo(x), \fzt(x))^T$ and sometimes write $D^{r}f=f^{(r)}$. We assume that the two eigenvalues of $Hf(x)$ are distinct. Then $V(x)$ and $\lambda(x)$ take the following forms $V(x)=G(d^2 f(x))$ and $\lambda(x)=J(d^2 f(x))$ for some function $G=(G_1, G_2)^T: \mathbb{R}^3 \to \mathbb{R}^2$ and $J: \mathbb{R}^3 \to \mathbb{R}$ given by
\begin{align}
G(u, v, w)&= \begin{pmatrix}
           2u-2w+2v- 2 \sqrt{ (w-u)^2+4 v^2 } \\
           w-u+4v- \sqrt{(w-u)^2+4 v^2 },
     \end{pmatrix} ,\\
     J(u,v,w)&=  \frac{1}{2}\left( u+ w- \sqrt{ (u-w)^2+4 v^2} \right) .
\end{align}
Throughout the proofs, we may take the normalized version of the eigenvector $V$, that is, $\|V\|=1$. This is not necessary but it simplifies discussion.

For some $a^*>0$, define $\mathcal{G}=\{\intgx^*(t):x_0 \in \mathcal{L}^*,-a^*\leq t\leq a^* \}$. We choose $a^*$ small so that $\mathcal{G} \subset [0,1]^2$ and $t^*_{x_0}$ is unique for any $x_{0} \in \mathcal{G}$. Define $\mathcal{L}^*\oplus \delta=\cup_{x \in \mathcal{L}^*}B(x,\delta)$, where $B(x,\delta)$ is an open ball around $x$ of radius $\delta$. The following assumptions will be needed for the theory.
\begin{itemize}[label={}]
\item (A1) The truth $f^*$ belongs to a H\"{o}lder Space $\mathcal{H}^{\alpha}([0,1]^2)$ with $\alpha \geq 4$.
\vspace{-2mm}
\item (A2) There is some $\delta>0$ small, such that for all $x \in (\mathcal{L}^*\oplus\delta) \cap [0,1]^2$, $H f^*(x)$ has two distinct eigenvalues, with smaller eigenvalue $\lambda^*(x)\leq -\eta$ for some small positive value $\eta$.
\vspace{-2mm}
\item (A3) For the $\delta>0$ in (A2), for all $x \in \mathcal{L}^* \oplus \delta$, $\vert \langle  \nabla \langle \nabla f^*(x), V^*(x)  \rangle, V^*(x)  \rangle \vert
 \geq\eta $ for the same positive value $\eta$ in (A2).
\vspace{-2mm}
\item (A4) The filament $\mathcal{L}^*$ is a compact set such that $\mathcal{L}^*=\{ \intgx^*(t^*_{x_0}):x_0 \in \mathcal{G} \}$.
\vspace{-2mm}
\item (A5) Assume that there exits some $C_{\mathcal{G}}>0$, for any $x_0 \in \mathcal{G}$,
\[
\inf_{x_0 \in \mathcal{G}} \inf_{t^*_{\xz}-a^*\leq s<u \leq t^*_{\xz}+a^*} \left\| \frac{\intgx^*(u)-\intgx^*(s)}{u-s} \right\|\geq C_{\mathcal{G}}.  \label{asp:lowerbound}
\]
\end{itemize}

Assumption (A2) is important for our analysis as it guarantees the smoothness of $V^*$ (ref. (3.1) of \citet{qiao2016annals} ) and that the normal vector of the filament $\nabla \langle \nabla f^*(x), V^*(x)  \rangle$ is well-defined. Similar assumptions are needed in \citet{genovese2014annals} and \citet{chen2015annals}.

Assumption (A3) says that the normal vector of the filament $  \nabla \langle \nabla f^*(x), V^*(x)  \rangle$ is not orthogonal to $V^*(x)$. In addition, it implies that  $\nabla \langle \nabla f^*(x), V^*(x)  \rangle \neq 0 $, i.e, the set $\{x:\langle \nabla f^*(x), V^*(x)  \rangle =0 \}$ is not ``thick". This means that a small change of $x \in \mathcal{L}^*$ necessarily changes the sign of $\langle \nabla f^*(x), V^*(x)  \rangle$. If one restricts attention to the locus defined by $ \langle \nabla f^*(x), V^*(x)  \rangle =0  $, noting rank$(\nabla \langle \nabla f^*(x), V^*(x)  \rangle)$=1, the implicit function theorem says that $\mathcal{L}$ is a one-dimensional manifold in $\mathbb{R}^2$. If for $x \in \mathcal{L}^*$, $ \nabla f^*(x) = 0$, Assumption (A3) should be interpreted as $\vert (V^*(x))^T Hf^*(x)V^*(x)\vert \geq \eta $ which is then implied by Assumption (A2). To see this, notice that
 \[
  \langle  \nabla \langle \nabla f^*(x), V^*(x)  \rangle, V^*(x)  \rangle  = \nabla f^*(x)^T \nabla V^*(x) V^*(x)+V^*(x)^THf^*(x)V^*(x) ,
  \]
but $V^*(x)^THf^*(x)V^*(x) =\lambda^*(x)$.
Assumption (A3) parallels the assumption (A2) in \citet{genovese2014annals}, where they assumed some upper bounds on the quantity related to the third derivative of the density function; see also the assumption (P1) in \citet{chen2015annals}.
Assumption (A5) is common in the literature; see \citet{koltchinskii2007annals, qiao2016annals}. Several useful consequences of these assumptions are summarized in Remark \ref{misc:integral_curves}. It is worth to point out that under these assumptions $Hf^*(x)$ must admit two distinct eigenvalues over some domain and $V^*(x)$ is Lipschitz continuous over this domain.

\section{Posterior contraction and credible sets for filaments }
In this section, we provide the main theoretical results. With a suitable choice of a series prior, the function $f$ from the posterior can induce its own integral curve, filaments and hitting time just as how they are defined previously. The goal is to establish posterior contraction rates of these objects relative to some metrics in a similar spirit of the current literature. Here is a brief outline of these results. Theorem 5.1 provides posterior contraction rates for the integral curve. Proposition 5.2 gives posterior contraction rates for the hitting time. Theorem 5.3 gives the posterior contraction rates for the filament along the integral curve. Theorem 5.8 establishes the posterior contraction rates for the filament around the truth, posterior rates for deviations between the posterior filament and the filament induced by the posterior mean, together with the convergence rates for the filament induced by posterior mean to that of the truth, all in terms of the Hausdorff distance. Theorem 5.9 provides a valid credible set with sufficiently high frequentist coverage.

To simply the presentation, we simply set  $J_1=J_2=J$ for some common positive value $J$ that grows with the sample size $n$. All results remain valid even when $J_1$ and $J_2$ are not equal but grow at the same rate, which is still denoted by the symbol $J$. 

The following result gives a Bayesian counterpart of Theorem 3.3 of \citet{qiao2016annals}. Our proof is similar to theirs, but a few major technical details are different.
\begin{theorem}  \label{rates:integral_curves}
Under Assumptions (A1),(A2) and (A5),  for $J_1=J_2=J \asymp (n/\log n)^{1/(1+2\alpha)}$, we have the following posterior contraction rate, for $\epsilon_n=(n/\log n)^{(2-\alpha)/(1+2\alpha)}$ and any $M_n\to \infty $,
\begin{equation}
\Pi(\sup_{x_0 \in \mathcal{G}}\sup_{t \in [t^*_{\xz}-a^*,t^*_{\xz}+a^*]}\|\intgx(t)-\intgx^*(t) \|>M_n\epsilon_{n}|\mathbb{D}_n) \xrightarrow{\mathrm{P}_0}0.
\end{equation}

\end{theorem}

\begin{remark} \label{rates:integral_curves_remark}
{\rm According to Assumption (H1) and the Theorem 3.3 of \citet{qiao2016annals}, their rate is $n^{-2/9}\sqrt{\log{n}} $ assuming that $\alpha=4$. Here with the choice $J \asymp (n/\log n)^{1/(1+2\alpha)}$, we obtain a rate which has a better logarithmic factor for the case $\alpha=4$, since then the rate reduces to $(n/\log n)^{-2/9}$. Note that if we choose $J \asymp (n/\log n)^{1/(2(1+\alpha))}$ as it is the optimal choice for estimation of the function $f^*$, it is easy to see from the proof that now  $\sup_{x}\|d^2f(x)-d^2f^*(x) \|^2$ is of order $(n/ \log n)^{(2-\alpha)/(1+\alpha)}$ and the term (\ref{integral:eqn:1}) in Section 8 has posterior contraction rate of order $ (n/\log n)^{(2-\alpha)/(1+\alpha)}$. Thus the contraction rate will then be $\epsilon_n=(n/\log n)^{(2-\alpha)/(2(1+\alpha))}$, which is ``suboptimal'' in the present context. }
\end{remark}

The following proposition is a Bayesian analog of Proposition A.1 and  Proposition 5.1 of \citet{qiao2016annals}. We can obtain better rates by using different magnitude of the tuning parameter; see the remark after the proposition.
\begin{proposition}  \label{rates:hitting_times}
Under Assumptions (A1)--(A5), for $J_1=J_2=J \asymp (n/\log n)^{1/(1+2\alpha)}$, we have the following posterior contraction rate: for $\epsilon_{n}=(n/ \log n)^{(5-2\alpha)/(2(1+2\alpha))}$ and any $M_n \to \infty$, \\
\begin{equation}
\Pi(\sup_{x_0 \in \mathcal{G}} |t^*_{x_0}-t_{x_0}| >M_n\epsilon_{n}|\mathbb{D}_n) \xrightarrow{\mathrm{P}_0}0.
\end{equation}
If in addition, $ \nabla f^*(x)=0$ for all $x\in \mathcal{L}^*$, then the rates improve to $\epsilon_{n}=(n/ \log n)^{(2-\alpha)/(1+2\alpha)} $.
\end{proposition}

\begin{remark}  \label{rates:hitting_times_remark}
{\rm
A better rate can be obtained if we choose $J \asymp (n/\log n)^{1/(2(1+\alpha))}$. With this choice, by Lemma \ref{posterior_rates_f},
$\sup_{x}\| \nabla f(x)-\nabla f^*(x)\| $ has posterior contraction rate $(n/ \log n)^{(1-\alpha)/(2(1+\alpha))}$, while $ \sup_{x}\| V(x)-V^*(x)\|$ has posterior contraction rate $ (n/ \log n)^{(2-\alpha)/(2(1+\alpha))}$. Recalling Remark \ref{rates:integral_curves_remark}, the posterior contraction rate will then be $\epsilon_{n}=(n/ \log n)^{(2-\alpha)/(2(1+\alpha))}$.  If in addition, $  \nabla f^*(x)=0$ for all $x\in \mathcal{L}^*$, the function then is a plateau without any rise or fall along the direction of filament.
In this case, the rate improves to $\epsilon_{n}=(n/ \log n)^{(1-\alpha)/(2(1+\alpha))} $. }
\end{remark}

Theorem \ref{rates:filaments_on_curves} below is a Bayesian analog of the convergence of filaments points on the integral curves starting from the same points; see \citet{qiao2016annals}, Section 3.4, for similar results. Again a better rate is possible; see the remark following the theorem.

\begin{theorem} \label{rates:filaments_on_curves}
Under Assumptions (A1)--(A5), for $ J_1=J_2=J \asymp  (n/\log n)^{1/(1+2\alpha)}$,  we have the following posterior contraction rates: for $\epsilon_{n}=(n/ \log n)^{(5-2\alpha)/(2(1+2\alpha))}$ and any $M_n\to \infty $,
\begin{equation}
\Pi(\sup_{x_0 \in \mathcal{G}} \|\intgx(t_{x_0})-\intgx^*(t_{x_0}^*) \|>M_n\epsilon_{n}|\mathbb{D}_n) \xrightarrow{\mathrm{P}_0}0.
\end{equation}
\end{theorem}

\begin{remark}  \label{rates:filaments_on_curves_remark}
{\rm
A better rate can be obtained with the choice $J \asymp (n/\log n)^{1/(2(1+\alpha))}$, giving $\epsilon_{n}=(n/ \log n)^{(2-\alpha)/(2(1+\alpha))}$ in view of Remark \ref{rates:integral_curves_remark} and Remark \ref{rates:hitting_times_remark}.  In particular, the degree of smoothness of the function $f$ has been accounted for, thanks to the series approximation. For $\alpha \geq 4,$ this is an improvement over the rate $n^{-1/6}\sqrt{\log{n}}$ in Theorem 3.4 of \cite{qiao2016annals} (see also their Proposition 5.1).  }
\end{remark}

In the following, we shall consider the Hausdorff distance between two filaments.
Given two sets $A$ and $B$ under Euclidean metric, let  $d(A|B):=\sup_{x \in A}\inf_{y \in B}\|x-y\|$. The Hausdorff distance between $A$ and $B$ is defined as
\[
\text{Haus}(A,B)=\max\{ d(A|B),d(B|A) \}.
\]
In what follows, we provide an upper bound for the Hausdorff distance and construct credible sets for the filaments. In fact, Theorem \ref{rates:filaments_on_curves} gives an upper bound for the Hausdorff distance. However, for the purpose of constructing credible sets with sufficient frequentist coverage, we need to have the upper bound in terms of more primitive quantities such as the derivatives of underlying function. In view of Remark \ref{rates:filaments_on_curves_remark} and the fact that integral curves are only intermediate objects, we henceforth restrict to the choice $J  \asymp  (n/\log n)^{1/(2(1+\alpha))}$.

 Recall that $\tilde{f}=A_{(0,0)}Y+C_{(0,0)}\theta_0$ is the posterior mean of $f$ conditional on $\mathbb{D}_n$ and that $\tilde{V}$,$ {\tilde{\Upsilon}_{x_0}}$, $\tilde{\mathcal{L}}$ are the corresponding eigenvector, integral curve and filament induced by $\tilde{f}$. In view of Lemma \ref{posterior_rates_fmean} and Lemma \ref{convergence_rates_f}, following the proofs of Theorem \ref{rates:integral_curves}, Proposition \ref{rates:hitting_times}, Theorem \ref{rates:filaments_on_curves}, it is straightforward to show that $\sup_{x_0 \in \mathcal{G}}\sup_{t}\|\intgx (t)-\tilde{\Upsilon}_{\xz}(t)\| $, $ \sup_{x_0 \in \mathcal{G}}\|t_{\xz}- \tilde{t}_{\xz}\| $ and $\sup_{x_0 \in \mathcal{G}}\|\Upsilon_{\xz}(t_{\xz})- \tilde{\Upsilon}_{\xz}(\tilde{t}_{\xz}) \|$ are all small with high posterior probability in $\mathrm{P}_0$-probability. Likewise, it can be shown that the quantities induced by $\tilde{f}$ converge to the corresponding true quantities induced by $f^*$. For instance $\tilde{\Upsilon}_{\xz}(t) $ converges to $\intgx^* (t)$ uniformly in $\xz \in \mathcal{G}$ and $t$, $\tilde{t}_{\xz}$ converges to $t^*_{\xz} $ uniformly in $\xz \in \mathcal{G}$ and $\tilde{\Upsilon}_{\xz}(\tilde{t}_{\xz}) $ converges to $\Upsilon^*_{\xz}(t^*_{\xz})$ uniformly in $\xz \in \mathcal{G}$ in $\mathrm{P}_0$-probability.

The following two theorems summarize above observations. Theorem \ref{twins:frequentist} is on the convergence rates of the Bayesian estimates of filaments to the true filaments. Theorem \ref{twins:bayesian} is on the posterior contraction rates of filaments around filaments induced by posterior mean.
  \begin{theorem} \label{twins:frequentist}
  	Under Assumptions (A1),(A2) and (A5),  for $J_1=J_2= J \asymp (n/\log n)^{1/(2(1+\alpha))}$, we have the following convergence rates:
  \begin{equation}
   	 \sup_{x_0 \in \mathcal{G}}\sup_{t \in [t^*_{\xz}-a^*,t^*_{\xz}+a^*]}\|\tilde{\Upsilon}_{\xz}(t)-\intgx^*(t) \|= O_p\left(\big(n/\log n\big)^{\frac{2-\alpha}{2(1+\alpha)}} \right).
  \end{equation}
  	If in addition, (A3) and (A4) hold, then
  	\begin{equation}
   	\sup_{x_0 \in \mathcal{G}} |\tilde{t}_{x_0}-t^*_{x_0}| =O_p\left(\big(n/\log n\big)^{\frac{2-\alpha}{2(1+\alpha)}} \right), 	
    \end{equation}
  	and
  	\begin{equation}
  	\sup_{x_0 \in \mathcal{G}} \|\tilde{\Upsilon}_{\xz}(\tilde{t}_{x_0})-\intgx^*(t_{x_0}^*) \| =O_p\left(\big(n/\log n\big)^{\frac{2-\alpha}{2(1+\alpha)}} \right) .
  	\end{equation}
  \end{theorem}

  \begin{theorem} \label{twins:bayesian}
  	Under Assumptions (A1),(A2) and (A5),  for $ J_1=J_2=J  \asymp (n/\log n)^{1/(2(1+\alpha))}$, we have the following posterior contraction rates: for any $M_n\to \infty $,
  	\begin{equation}
  	\Pi(\sup_{x_0 \in \mathcal{G}}\sup_{t \in [t^*_{\xz}-a^*,t^*_{\xz}+a^*]}\|\intgx(t)-\tilde{\Upsilon}_{\xz}(t) \|>M_n (n/\log n)^{(3-2\alpha)/(4(1+\alpha))} |\mathbb{D}_n) \xrightarrow{\mathrm{P}_0}0.
    \end{equation}
  	If in addition, (A3) and (A4) hold, then
  	\begin{equation}
  	\Pi(\sup_{x_0 \in \mathcal{G}} |t_{x_0}-\tilde{t}_{x_0}| >M_n (n/ \log n)^{(2-\alpha)/(2(1+\alpha))}|\mathbb{D}_n) \xrightarrow{\mathrm{P}_0}0,
  	\end{equation}
  	and
  	\begin{equation}
  	\Pi(\sup_{x_0 \in \mathcal{G}} \|\intgx(t_{x_0})-\tilde{\Upsilon}_{\xz}(\tilde{t}_{x_0}) \|>M_n  (n/ \log n)^{(2-\alpha)/(2(1+\alpha))}|\mathbb{D}_n) \xrightarrow{\mathrm{P}_0}0.
  	\end{equation}
  \end{theorem}

The following proposition says that with high posterior probability the induced filament from the posterior satisfies similar properties the true filament has,  so does the induced filament by $\tilde{f}$, with $\mathrm{P}_0$-probability tending to one.

\begin{proposition} \label{twins}
Suppose that $f$ has a tensor-product B-splines prior with order $q_1=q_2\geq\alpha$ and $J \asymp (n/\log n)^{1/(2(\alpha+1))}$, then the following assertions hold.
(i) The filament $\mathcal{L}$ of $f$ drawn from the posterior distribution satisfies assumptions (A2)--(A5) with posterior probability tending to $1$ under $\mathrm{P}_0$-probability;
(ii) The induced filament $\mathcal{\tilde{L}}$ of the posterior mean $\tilde{f}$ satisfies assumptions (A2)--(A5) with $\mathrm{P}_0$-probability tending to $1$.
\end{proposition}

The following lemma is inspired by \citet{genovese2014annals} Theorem 4, where they relate the Hausdorff distance between filaments to the Euclidean distance between $V$ but under a different set of assumptions.

\begin{lemma} \label{bounds:hausdorff}
Consider for two regression functions $f$ and $\hat{f}: [0,1]^2 \mapsto \mathbb{R}$ that are sufficiently close in supremum metric and both satisfy assumptions (A1)--(A5), then the Hausdorff distance between the two induced filaments satisfies, for some positive constant $c_1$,
\begin{equation}
\mathrm{Haus}(\mathcal{L},\mathcal{\hat{L}})
\leq \frac{c_1}{\eta}\Bigl( \|f^{(2,0)}-\hat{f}^{(2,0)}\|_{\infty}+  \|f^{(1,1)} -\hat{f}^{(1,1)}\|_{\infty} + \|f^{(0,2)}-\hat{f}^{(0,2)} \|_{\infty} \Bigr).
\end{equation}
\end{lemma}

In view of above lemma, Theorems \ref{twins:frequentist}, \ref{twins:bayesian} and Proposition \ref{twins}, we have the following result.

\begin{theorem}\label{rates:hausdorff}
Under Assumptions (A1)--(A5), with the choice of $J_1=J_2=J \asymp (n/\log n)^{1/(2(\alpha+1))}$,  for any $M_n \to \infty$,
\begin{equation}
	\Pi(\mathrm{Haus}(\mathcal{L},\mathcal{{L}^*})>M_n (n/\log n)^{\frac{2-\alpha}{2(\alpha+1)}}|\mathbb{D}_n )\xrightarrow{\mathrm{P}_0}0,
\end{equation}
\begin{equation}
	\Pi(\mathrm{Haus}(\mathcal{L},\mathcal{\tilde{L}})>M_n (n/\log n)^{\frac{2-\alpha}{2(\alpha+1)}}|\mathbb{D}_n )\xrightarrow{\mathrm{P}_0}0,
\end{equation}
and
\begin{equation}
	 \mathrm{Haus}(\mathcal{\tilde{L}},\mathcal{{L}^*})= O_p \big((n/\log n)^{\frac{2-\alpha}{2(\alpha+1)}} \big).
\end{equation}
\end{theorem}

\begin{remark}\label{rates:hausdorff_remark}
{\rm
The third assertion of Theorem \ref{rates:hausdorff} for the convergence rate of the filament induced by the posterior mean is an improvement over the rate $(\log n /n)^{1/5}$ in Theorem 5 of \citet{genovese2014annals} when $\alpha> 4$. }
\end{remark}

For the following result, we restrict multi-index $r$ to the collection $\mathcal{R}:=\{(2,0),(1,1),(0,2)\}.$ Let $\tilde{f}^{(r)}:=A_rY+C_r\theta_0$ be the posterior mean of $f^{(r)}$ and $\mathcal{\tilde{L}}$ be the induced filament. For some $0<\gamma<1/2$, let $R_{n,r,\gamma}$ denote the $1-\gamma$ quantile of the posterior distribution of $\|f^{(r)}-\tilde{f}^{(r)}\|_{\infty}$. Let $C_{f,r,\gamma}^{\rho}:=\{ f: \|f^{(r)}-\tilde{f}^{(r)}\|_{\infty} \leq\rho R_{n,r,\gamma} \}$ for some large $\rho>1$. The following theorem provides two valid credible sets with sufficiently high frequentist coverage as the sample size increases. The choices of $c_1$, $\eta$ and $\rho$ will be discussed in the next section.

\begin{theorem} \label{credible_sets}
 Assume (A1)--(A5), for the choice of $J_1=J_2=J \asymp (n/\log n)^{1/(2(\alpha+1))}$ and some sufficiently large constant $\rho>1$, for the following two sets,
\begin{align}
&C_{\mathcal{L}}=\{ \mathcal{L}(f): f\in \cap_{r\in \mathcal{R}} C_{f,r,\gamma}^{\rho} \} ,\\
&\bar{C}_{\mathcal{L}}=\{ \mathcal{L}: {\rm Haus}(\mathcal{L},\mathcal{\tilde{L}})\leq \frac{c_1}{\eta}\rho \max_{r \in \mathcal{R}} R_{n,r,\gamma} \},
\end{align}
the credibility of $C_{\mathcal{L}}$ and its coverage probability for $\mathcal{L^*}$ tend to $1$ and $C_{\mathcal{L}} \subset \bar{C}_{\mathcal{L}}$ with high posterior probability with $\mathrm{P}_0$-probability tending to $1$.
\end{theorem}

\section{Simulation}
Many algorithms have been proposed to find filaments. We here use an algorithm that shares a similar spirit of the Subspace Constrained Mean Shift (SCMS) algorithm proposed in \citet{ozertem2011locally}. SCMS algorithm was also used in \citet{genovese2014annals}, \citet{chen2015annals} and \citet{chen2015astronomy}. The key of the algorithm is to project the gradient onto the direction given by $V$. Even though in the literature the algorithm is primarily used with kernel density estimator, our study suggests that nothing hinders the efficacy of the algorithm when applied with a series based estimation in either regression or density estimation setting. In the following, we give a description of the algorithm.
\\
\noindent\rule{14cm}{1pt}

\noindent \textbf{Algorithm}: (Subspace Constrained Gradient Ascent Algorithm)\\
Set $\epsilon>0$, $\tau>0$, $\bar{a}>0$ and select a collection of points $\{x_1,\ldots,x_n\}$, compute $f(x_i)$ and keep only those points for which $f(x_i)>\tau$. For each $x_i$, let $x_i^{(1)}=x_i$. Now iterate through the following  steps starting from $t=1$:
\begin{enumerate}[label={(\arabic*)}]
	\item  evaluate $\nabla f(x_i^{(t)})$;
	\vspace{-0mm}
	\item  evaluate the Hessian $Hf(x_i^{(t)})$ and perform spectral decomposition to get $V(x_i^{(t)})$ the
	normalized \newline eigenvector of $Hf(x_i^{(t)})$ with the smallest eigenvalue;
	\vspace{-0mm}
	\item  update $x_i^{(t+1)}=\bar{a} V(x_i^{(t)})V^T(x_i^{(t)})   \nabla f(x_i^{(t)})+x_i^{(t)}$;
	\vspace{-0mm}
	\item  stop if $\|x_i^{(t+1)}-x_i^{(t)}\| < \epsilon$ or $|V^T(x_i^{(t)}) \nabla f(x_i^{(t)}) | < \epsilon$.
\end{enumerate}
If stop at $t=t^*$, keep the point $x^{(t^*)}_i$ that satisfies  $\lambda(x_i^{(t^*)}) <0$.

\vspace{-2mm}
\noindent\rule{14cm}{1pt}

\begin{figure}[H]
	\begin{minipage}[b]{1\linewidth}
		\centering
		\includegraphics[width=.4\linewidth]{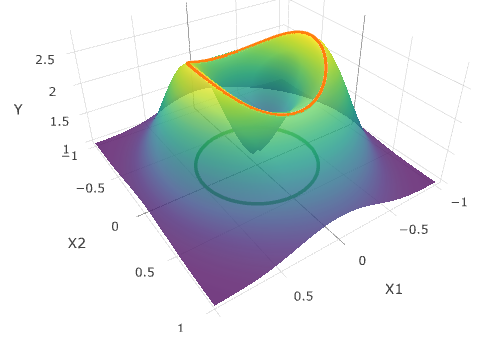}
	\end{minipage}
	\caption{The function $f$ and its filament. }	
	\label{fig:function}
\end{figure}

\begin{figure}[H]
	\begin{minipage}[b]{0.45\linewidth}
		\centering
		\includegraphics[width=.7\linewidth]{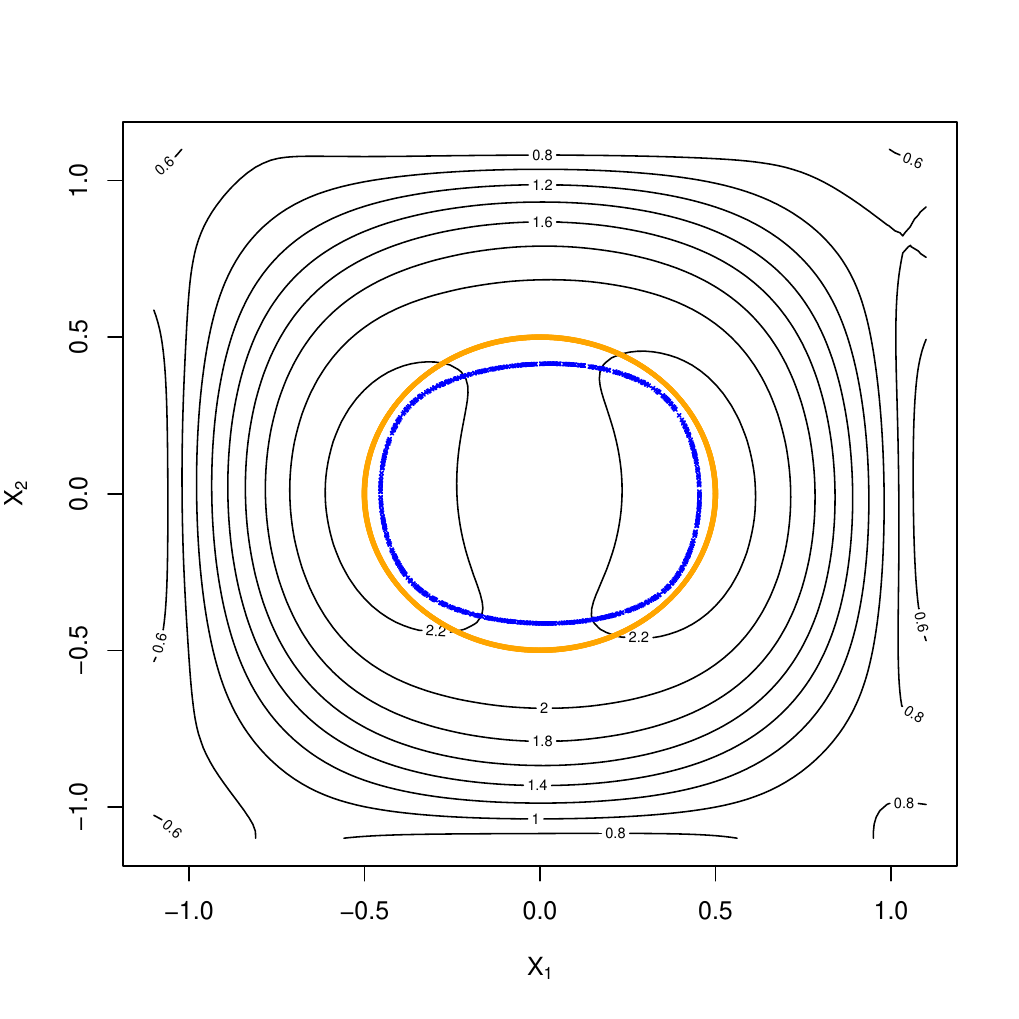}
	\end{minipage}
	\begin{minipage}[b]{0.45\linewidth}
		\centering
		\includegraphics[width=.7\linewidth]{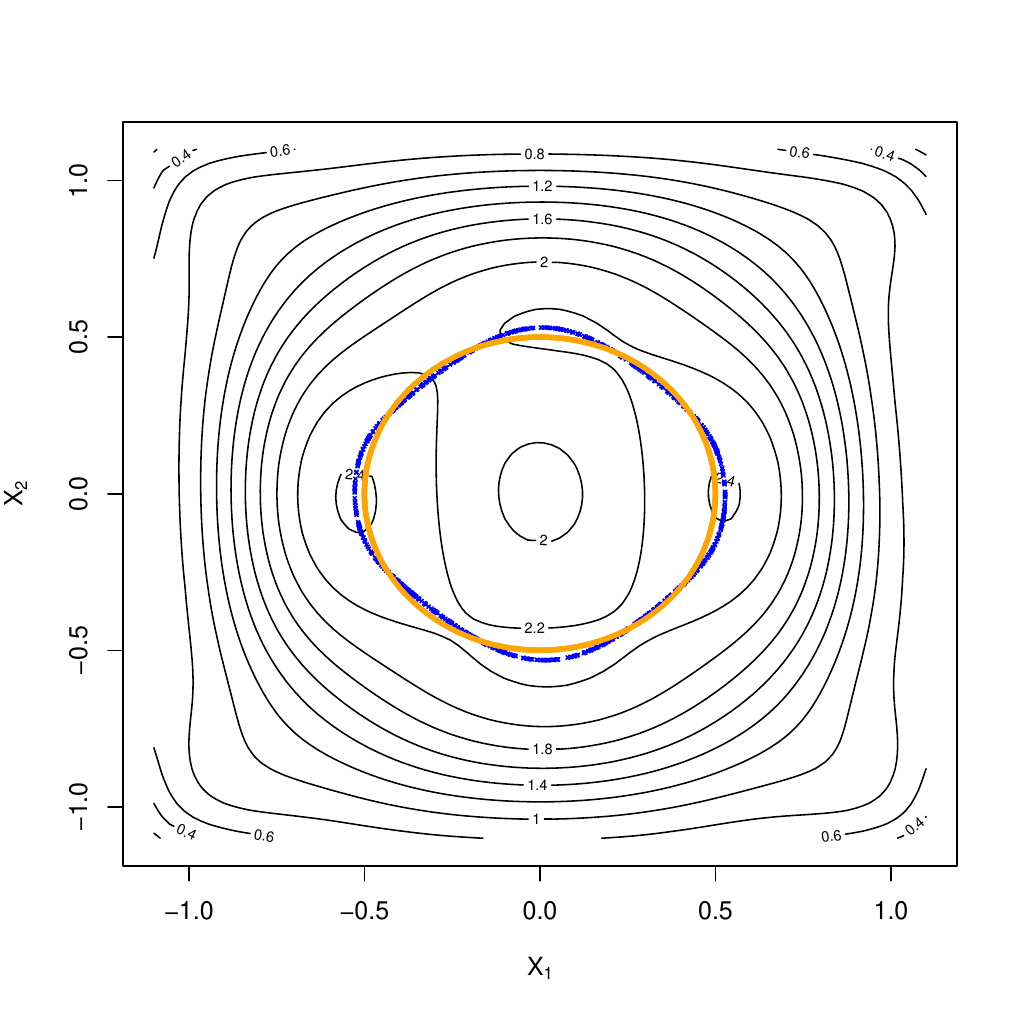}
	\end{minipage}
    \\
	\begin{minipage}[b]{0.45\linewidth}
		\centering
		\includegraphics[width=.7\linewidth]{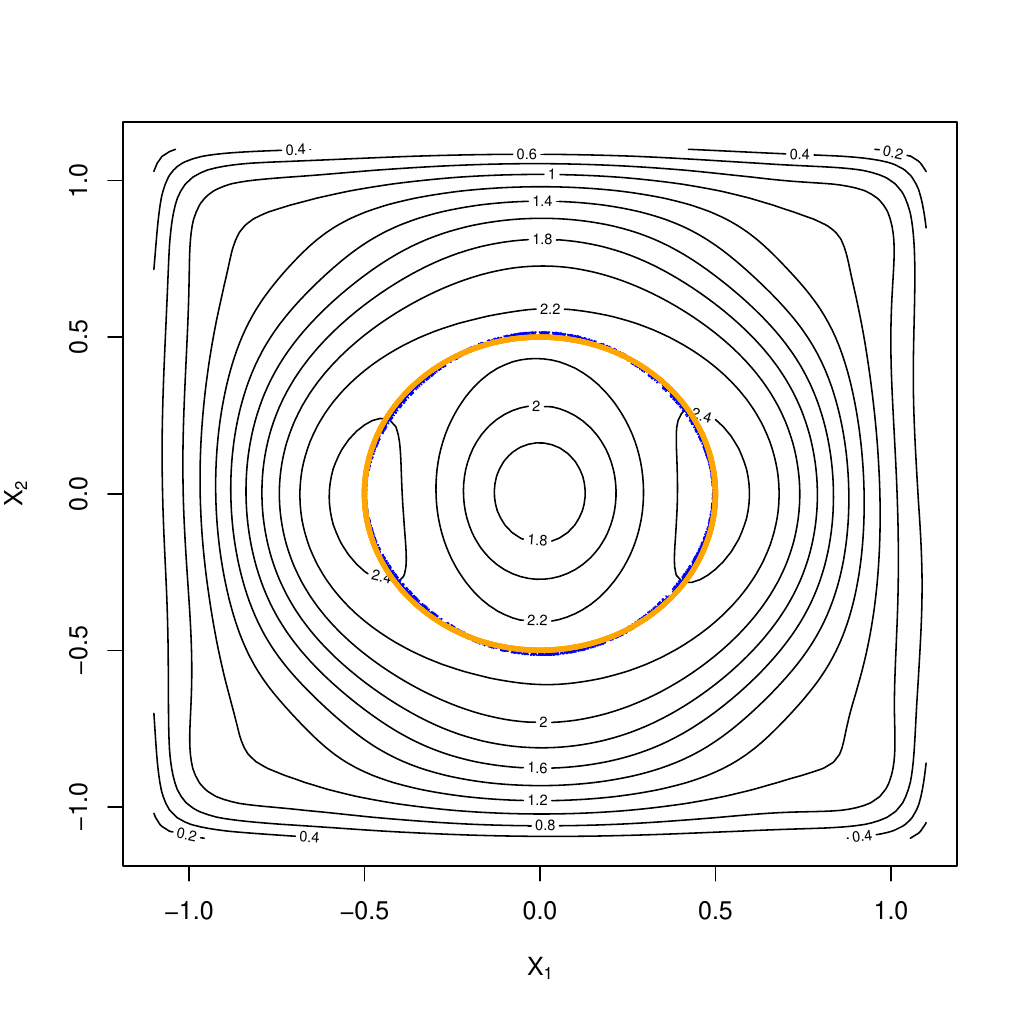}
	\end{minipage}
	\begin{minipage}[b]{0.45\linewidth}
		\centering
		\includegraphics[width=.7\linewidth]{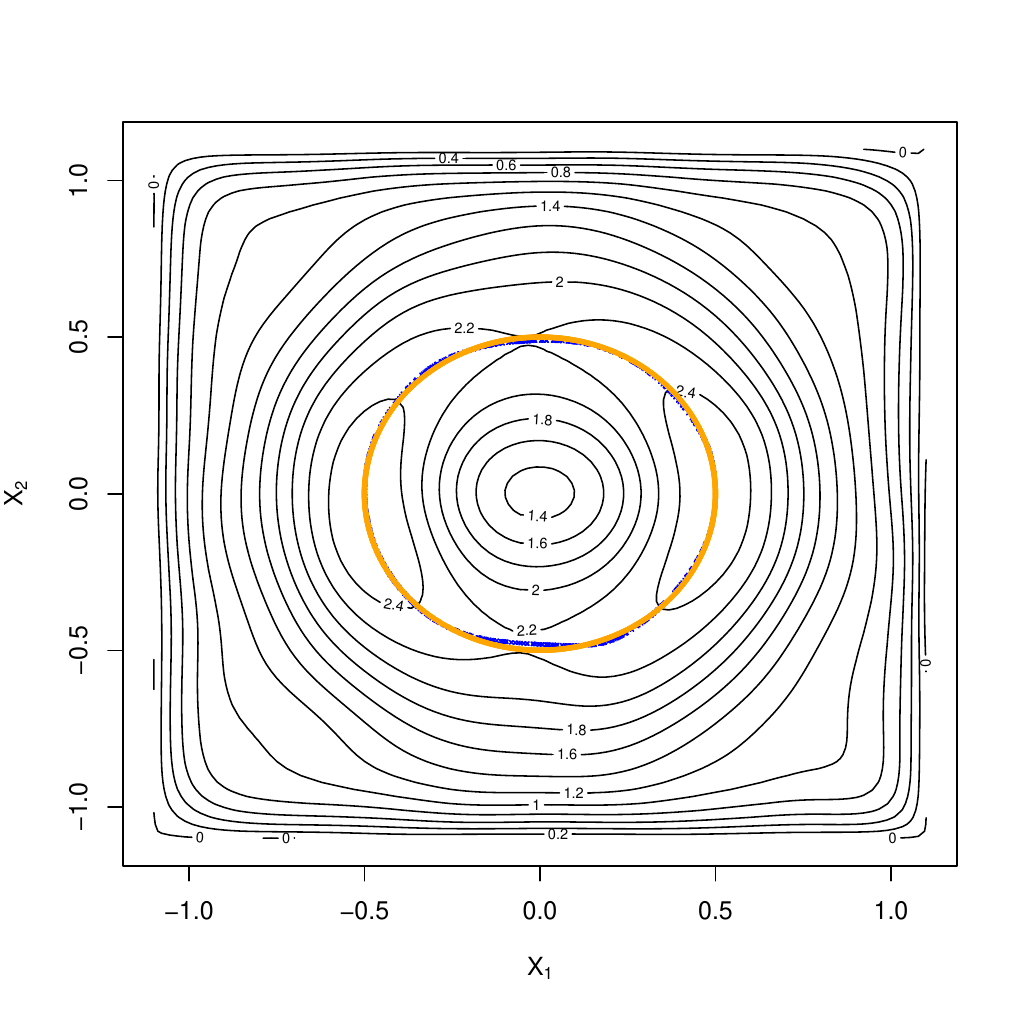}
	\end{minipage}
	\caption{Effects of the smoothing parameter $J_1$ and $J_2$. The orange circle is the truth, blue curve (consists of dots) is estimated filament induced by posterior mean. Top left: $J_1=J_2=7$; top right: $J_1=J_2=8$; bottom left: $J_1=J_2=9$; bottom right: $J_1=J_2=15$. }	
	\label{fig:tuning}
\end{figure}

In the simulation, we consider the following function
\[
f(x_1,x_2)=1+\left( \phi \big(\sqrt{x_1^2+x_2^2}\big) \right)^{1+\cos^2(\tan^{-1}(x_2/x_1))},
\]
where $\phi(\cdot)$ is the normal density function with mean $0.5$ and standard deviation $0.3$. We generate i.i.d. data $X_i$ uniformly on $[0,1]\times[0,1]$ and i.i.d. $\varepsilon_i$ from normal with mean 0 and standard deviation $0.1$ and then set $Y_i=f(X_{i,1},X_{i,2})+\epsilon_i$. The sample size is $2000$. The Figure \ref{fig:function} shows function $f$ and its filament.
The smoothing parameter $(J_1, J_2)$  plays an important role as in any nonparametric estimation problem.

We use fifth-order B-splines functions, that is $q_1=q_2=5$. One can choose the pair $(J_1,J_2)$ by their posterior mode (in a logarithmic scale) by maximizing the following
\[
\log \Pi(J_1,J_2|\mathbb{D}_n)= -2n \log\hat{\sigma}-\log(\det(B\Lambda_0B^T )+I_n)+ \text{const.}
\]

We use $\tau=2$, $\bar{a}=0.02$ and $\epsilon=10^{-6}$. Some pilot simulation suggests that $(J_1, J_2)=(9, 9)$ is picked by its posterior mode throughout. Different choices have been experimented as well. In general, ``oversmoothing" may distort the filaments, whereas ``undersmoothing" seems to produce similar results as using $J_1=J_2=9$, as can be seen in Figure \ref{fig:tuning}. We also provide uncertainty quantification in Figure \ref{fig:uncertainty} with $\gamma=0.1, \rho=1.2$. Each graph shows $100$ posterior filaments drawn from $C_{\mathcal{L}}$. To evaluate $R_{n,r,\gamma}$ for $r \in\mathcal{R}=\{(2,0),(1,1),(0,2)\},$ we first draw $200$ posterior samples of $\theta$, compute their posterior mean $\tilde{\theta}$. Next we compute $\sup_{x}|{b^{(r)}_{J_1,J_2}(x)}^T(\theta-\tilde{\theta})|$ by searching on a crude grid and pick the maximum point on the grid and then starting from this maximum point apply gradient ascent or descent method to check if nearby points can achieve greater (absolute) value. We keep the largest value as the supremum. The $(1-\gamma)$-empirical quantile over all these suprema will then be our $R_{n,r,\gamma}$. The filaments from the  posterior samples that fall in the set $C_{\mathcal{L}}$ can then be generated.

\begin{figure}[H]
	\begin{minipage}[b]{0.45\linewidth}
		\centering
		\includegraphics[width=.8\linewidth]{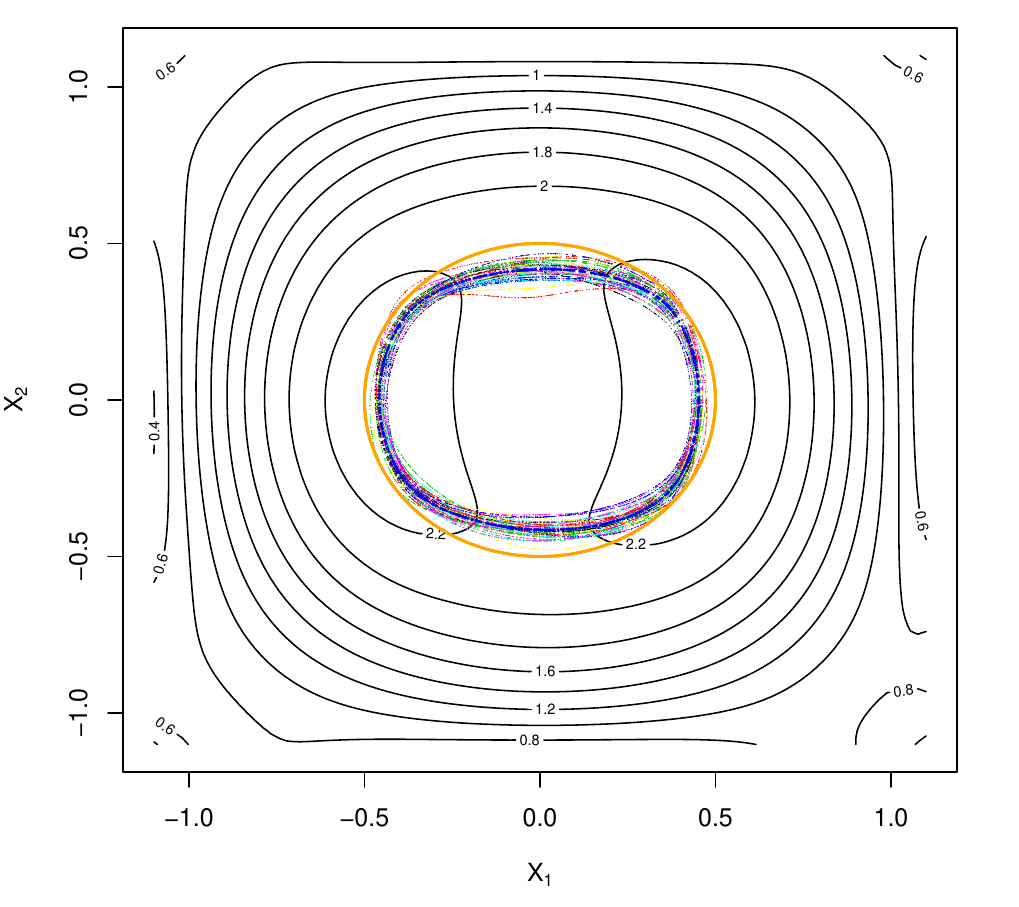}
	\end{minipage}
	\begin{minipage}[b]{0.45\linewidth}
		\centering
		\includegraphics[width=.8\linewidth]{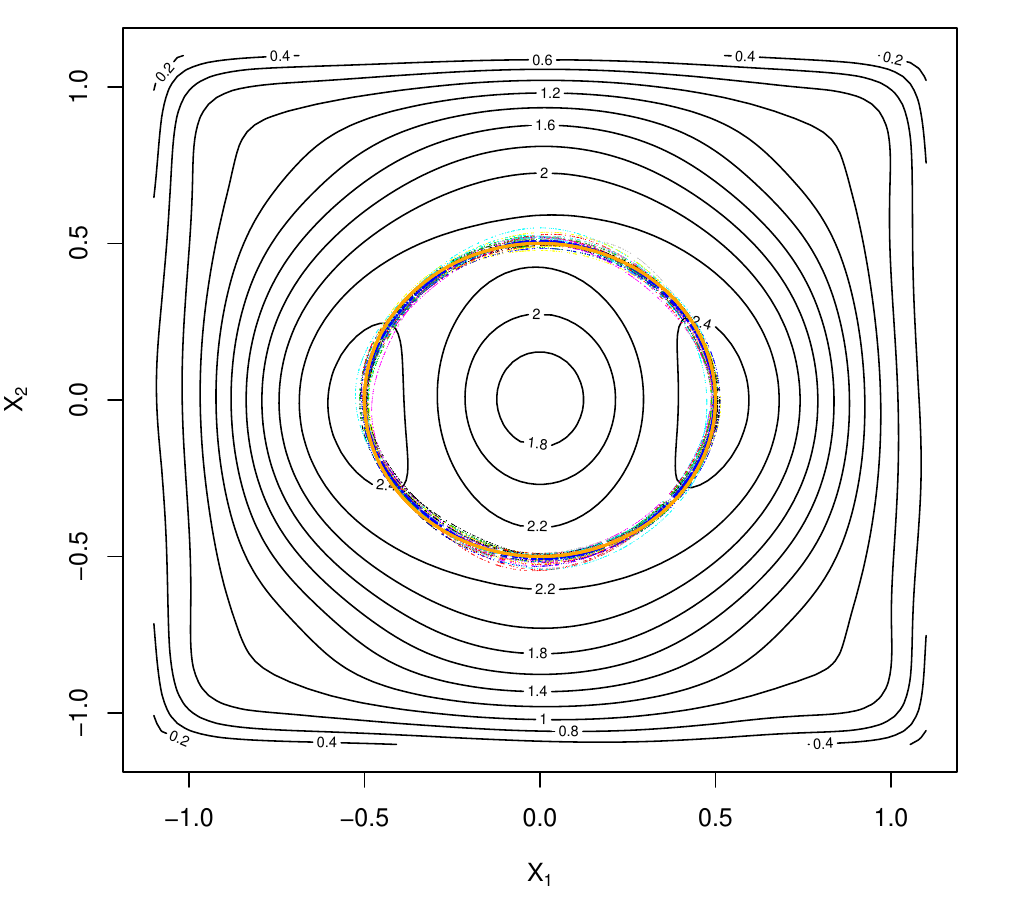}
	\end{minipage}	
	\caption{Uncertainty Quantification. Left: $J_1=J_2=7$. Right: $J_1=J_2=9$. }	
	\label{fig:uncertainty}
\end{figure}

To assess the performance over $100$ iterations, we compute the Hausdorff distance between $\mathcal{L}^*$ and $\tilde{\mathcal{L}}$. The value of $\rho$ should not be too large but gives a reasonable high percentage of $f$ from posterior that fall in $C_{f,r,\gamma}^{\rho}$. In practice, reasonable values of $\rho$ can be calibrated by some pilot simulation using the posterior samples. For instance, $\rho=1.2$ is a reasonable choice in this simulation, giving $92.33\%$ credibility averaging over all iterations. To evaluate the coverage performance, we compute the Hausdorff distance between $\mathcal{L}^*$ and $\tilde{\mathcal{L}}$. From the definition of $\bar{C}_{\mathcal{L}}$, we set $c_1/\eta$ to different values. Simulation shows that $\mathcal{L}^*$ belongs to $\bar{C}_{\mathcal{L}}$ for $ 91\%, 94\%, 98\%$ time when $c_1/\eta$ takes value $7.3\times 10^{-4}, 7.5\times 10^{-4}$ and $8 \times 10^{-4} $ respectively. In practice, $c_1$ can be computed using $\sup_{x}\| \nabla \tilde{f}(x)\|$ and $\eta$ as the smallest value of $-\lambda$ along the filament induced by the posterior mean. With this method, we obtain $100 \%$ coverage --- high coverage as the theory predicts.

\section{Application}
For application, we use an earthquake dataset for California and its vicinity from January 1st of 2013 to December 31th of 2017 with magnitude 3.0 and above on the Richter scale \footnote{The data is publicly available from \url{https://earthquake.usgs.gov/earthquakes/}. }. The dataset consists of 3772 observations, among which 3383 observations have magnitude between $3$ and $4$; $355$ observations between $4$ and $5$; $34$ observations above $5$. The average magnitude is $3.439$. The left panel in Figure \ref{fig:application} shows the data scatter plot. The sizes of circles are proportional to the magnitudes of the earthquakes.

In the algorithm, we use $\bar{a}=5\times 10^{-6} $ and $\tau=3$ and $\epsilon = 10^{-6}$. We use $q_1=q_2=4$ and $J_1=J_2=32$. We draw $200$ posterior samples to compute the posterior mean. The filaments induced by the posterior mean is plotted as the blue curve in Figure \ref{fig:application}. The same filaments are overlayed on the magnitude surface as given in Figure \ref{fig:application2}. To obtain uncertainty quantification using filaments from posterior samples, we use $\gamma=0.1$ (i.e., $90\%$ credibility), $\rho=1.2$. The results showed that there are $91 \%$ of posterior realizations fall into $C_{\mathcal{L}}$, only slightly higher than the nominal credibility level. We randomly pick $100$ of them to describe the uncertainty quantification as given in the right panel of Figure \ref{fig:application}.

\begin{figure}[H]
	\begin{minipage}[b]{0.5\linewidth}
		\centering
		\includegraphics[width=.7\linewidth]{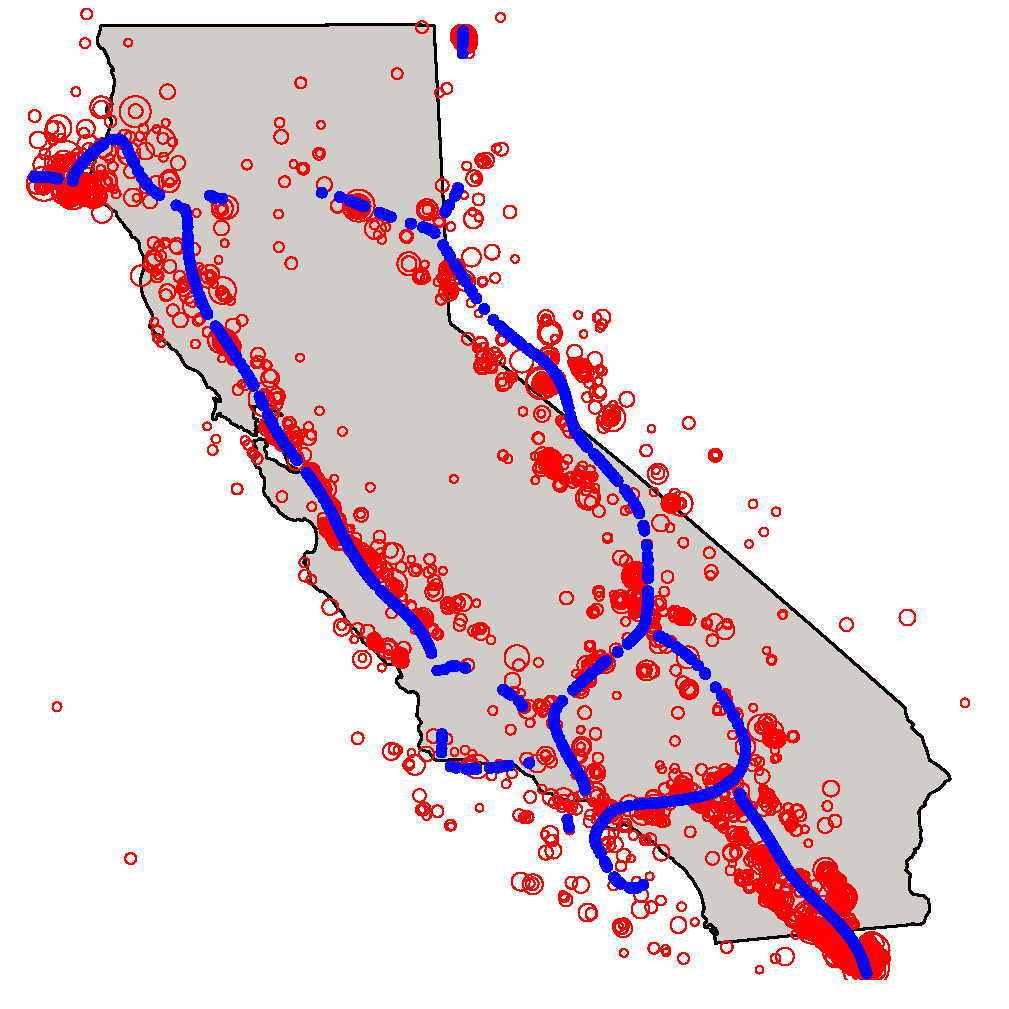}
	\end{minipage}
	\begin{minipage}[b]{0.5\linewidth}
		\centering
		\includegraphics[width=.7\linewidth]{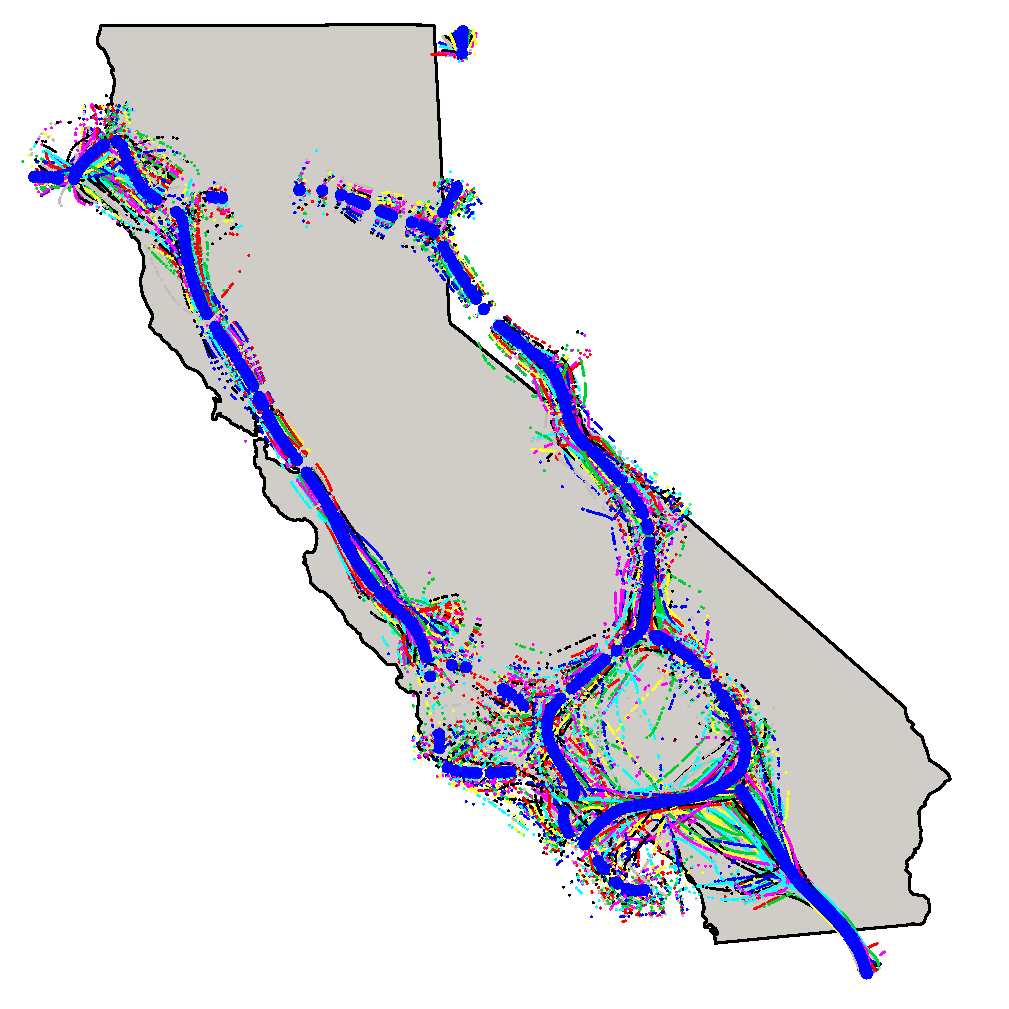}
	\end{minipage}
	\caption{ Left: Earthquake data points (in red). Right: 100 filaments in the posterior constructed with high frequentist coverage. The thick blue curves are the filaments induced by the posterior mean.}	
	\label{fig:application}
\end{figure}

\begin{figure}[H]
	\begin{minipage}[b]{1\linewidth}
		\centering
		\includegraphics[width=0.4\linewidth]{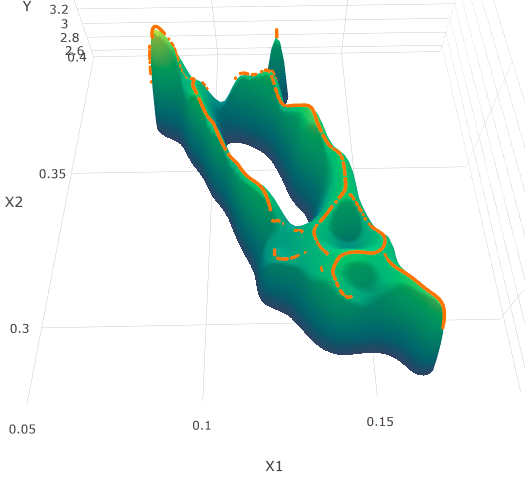}
	\end{minipage}
	\caption{The magnitude surface and filaments induced by the posterior mean (traces on the graph). }	
	\label{fig:application2}
\end{figure}

The filaments hence obtained provide useful characterization of the features of earthquake magnitude. Geographically, these filaments pass through the most populous coastal urban and suburban areas in California, for instance, Eureka city, San Francisco and Los Angeles. Since this application has utilized very small portion of earthquake data, we believe that a large scale study for different periods of historical times will be useful for the study of the dynamics of the earthquakes. Uncertainty quantification provides the statistical understanding of what might be considered as reasonable shifts of these filaments through spatial and temporal domain and are also helpful for discovering newly emerging crustal activities.

\section{Technical Proofs}
In this section, we will provide lemmas and formal proofs of the results stated in the main text. To focus on discussion, we will draw on several useful results directly given in the following remark.

\begin{remark} \label{misc:integral_curves}
{\rm Under the assumptions (A1) to (A5), the following assertions hold. }
\begin{enumerate}[label={(\arabic*)}]
	\item Integral curves are dense and non-overlapping as starting points vary; the set $\mathcal{G}$ is compact.
	\item $\xz \mapsto t^*_{\xz}$ is continuous.
	\item $\| \intgx^* (t)- \Upsilon^*_{x'_0} (t)  \| \leq C \exp(t)\| \xz- x'_0\|$ for some constant $C>0$.
	\item $G=(G_1, G_2)^T$,  $\nabla G$ are Lipschitz continuous and each element of Hessian $HG_1$, $HG_2$ is bounded on some open set $Q_{\delta} \subset \mathbb{R}^3$ such that $\{d^2f^*(x): x \in [0,1]^2 \} \subset Q_{\delta}$. Thus $V^*$ is Lipschitz continuous.
\end{enumerate}

\end{remark}
There results can be proved using arguments similar to \citet{qiao2016annals}. The original argument and proofs appear in that paper in various places. To save space, we do not provide the details here, but point to these following specific places in their paper. Result $(1)$ is given in the discussion in page 10, result $(2)$ and $(3)$ are discussed in pages 22 and 48, and result $(4)$ is discussed in pages 53--55.

We shall also use the following result frequently in our proofs.
\begin{remark}
	If the Condition \eqref{empirical:eqn} holds, then we have
	\[
	C_1n\left( \prod_{k=1}^2 J_k^{-1}\right)I_{J_1 J_2} \leq B^TB\leq C_2n\left( \prod_{k=1}^2 J_k^{-1}\right) I_{J_1 J_2},
	\]
	for some constants $C_1, C_2>0$.
	If in addition,  for some constants $0< c_1\leq c_2 <\infty$ such that
	\[
	c_1I_{J_1 J_2} \leq \Lambda \leq c_2I_{J_1 J_2},
	\]
	we have
	\[
	\left( C_1 n \prod_{k=1}^2 J_k^{-1}+ c_2^{-1}\right)I_{J_1 J_2} \leq B^TB+ \Lambda^{-1} \leq \left(  C_2n \prod_{k=1}^2 J_k^{-1}+c_1^{-1}\right) I_{J_1 J_2}.
	\]
\end{remark}
\begin{proof}
	See \citet{yoo2016supremum} Lemma A.9. and the discussion in p.1075.
\end{proof}

\subsection{Some lemmas}

\begin{lemma} \label{approx_error}
	For $f \in \mathcal{H}^{\alpha}([0,1]^2)$ with $\alpha \leq \min(q_1,q_2)$. Then for $J_1, J_2$ sufficiently large, there exists a function $f^\infty:=b_{J_1,J_2}^T \theta_{\infty}$ for some $\theta_\infty$ such that
	\begin{align*}
		\|f^\infty-f\|_{\infty}\leq C\sum_{k=1}^2 J_k^{-\alpha}, \qquad \qquad  \|D^{r}f^\infty-D^{r}f\|_{\infty}\leq C\sum_{k=1}^2 J_k^{r_k-\alpha},
	\end{align*}
	for some positive constant $C$ depending on $\alpha$ and $q$ and for every integer vector $r$ satisfying $|r|<\alpha$.
\end{lemma}

\begin{proof}
 For proof, see \citet{schumaker2007spline}.
\end{proof}

\begin{lemma} (Posterior contraction around truth) \label{posterior_rates_f}
	Under the above assumptions and $J_1 =J_2 =J$ and $J^2 \leq n$, with $J$ chosen such that $\delta_{n,k,J}$ vanishes to zero, then for any $M_n\to \infty $,
	\begin{align*}
		&\Pi(\sup_{x \in [0,1]^2} | D^{r}f(x)-D^{r}f^*(x)|>M_n\delta_{n,|r|,J}|\mathbb{D}_n)\xrightarrow{\mathrm{P}_0}0 , \\
		&\Pi( \sup_{x \in [0,1]^2}\| \nabla f(x)-\nabla f^*(x)\|>M_n \delta_{n,1,J}|\mathbb{D}_n)\xrightarrow{\mathrm{P}_0}0, \\
		&\Pi( \sup_{x \in [0,1]^2} \| d^2f(x)-d^2f^*(x)\|>M_n\delta_{n,2,J}|\mathbb{D}_n)\xrightarrow{\mathrm{P}_0}0 ,\\
		&\Pi( \sup_{x \in [0,1]^2} \|Hf(x)-Hf^*(x)\|_F>M_n\delta_{n,2,J}|\mathbb{D}_n)\xrightarrow{\mathrm{P}_0}0, \\
		&\Pi( \sup_{x \in [0,1]^2} \|  \nabla d^2f(x)- \nabla d^2f^*(x)\|_F>M_n \delta_{n,3,J}|\mathbb{D}_n)\xrightarrow{\mathrm{P}_0}0, \\
		&\Pi( \sup_{x \in [0,1]^2} \|V(x)- V^*(x)\|_F>M_n \delta_{n,2,J}|\mathbb{D}_n)\xrightarrow{\mathrm{P}_0}0, \\
		&\Pi( \sup_{x \in [0,1]^2} \| \nabla V(x)-\nabla V^*(x) \|_F>M_n \delta_{n,3,J}|\mathbb{D}_n)\xrightarrow{\mathrm{P}_0}0.
	\end{align*}
	where $\delta_{n,k,J}= J^{k}\bigl( (\log n/ n) J^2+J^{-2\alpha} \bigr)^{1/2}$.
\end{lemma}
\begin{proof}
	These results can be directly adapted from \citet{yoo2016supremum} by noting the highest degree of derivatives in each expression. We shall show the rates for $\sup_{x}\| H f(x) -Hf^*(x) \|_F$, $\sup_{x}\| \nabla d^2f(x) -\nabla d^2f^*(x) \|_F$ and $ \sup_{x}\| \nabla V(x)-\nabla V^*(x) \|_F$.
	Notice that
	\begin{align*}
	\| H f(x) -Hf^*(x) \|_F&=\Big(   (f^{(2,0)}(x)-f^{*(2,0)}(x))^2 + 2(f^{(1,1)}(x)-f^{*(1,1)}(x))^2+ (f^{(0,2)}(x)- f^{*(0,2)}(x) )^2 \Big)^{1/2}, \\
	&\leq 4^{-1/2} \big(|f^{(2,0)}(x)-f^{*(2,0)}(x)| + 2|f^{(1,1)}(x)-f^{*(1,1)}(x)|+ |f^{(0,2)}(x)-f^{*(0,2)}(x)|\big).	
	\end{align*}
	The rate for  $\sup_{x}\| H f(x) -Hf^*(x) \|_F$ then follows easily.
	
	Also, the contraction rates for $\| \nabla d^2f(x) -\nabla d^2f^*(x) \|_F $ follows from the inequality $\| \nabla d^2f(x) -\nabla d^2f^*(x) \|_F \leq  6^{-1/2} \sum_{r: |r|=3} |D^{r}f(x)-D^{r}f^*(x)| $.
	
	Lastly, since $\nabla V(x)-\nabla V^*(x) =\nabla G(d^2f(x)) \nabla d^2f(x)-\nabla G(d^2f^*(x)) \nabla d^2f^*(x)$, which is 2 by 2 matrix. Straightforward calculation gives its $(1,1)$ element
	\begin{align*}
		&\left(G_1^{(1,0,0)}(d^2f(x)) f^{(3,0)}(x)-G_1^{(1,0,0)}(d^2f^*(x)) f^{*(3,0)}(x)\right), \\ +& \left(G_1^{(0,1,0)}(d^2f(x)) f^{(2,1)}(x)-G_1^{(0,1,0)}(d^2f^*(x)) f^{*(2,1)}(x)\right), \\
		+&\left(G_1^{(0,0,1)}(d^2f(x)) f^{(1,2)}(x) -G_1^{(0,0,1)}(d^2f^*(x)) f^{*(1,2)}(x) \right)	.
	\end{align*}
	The absolute value of the first summand is bounded by the sum of
	\begin{align*}
	&\left| G_1^{(1,0,0)}(d^2f(x)) f^{(3,0)}(x)-  G_1^{(1,0,0)}(d^2f(x)) f^{*(3,0)}(x)\right| ,
	\end{align*}
	and
	\begin{align*}
    &\left|G_1^{(1,0,0)}(d^2f(x)) f^{*(3,0)}(x)- G_1^{(1,0,0)}(d^2f^*(x)) f^{*(3,0)}(x)\right|.
	\end{align*}
	Noting that $d^2f(x)$ contracts to $d^2f^*(x)$ uniformly in $x$, hence $\{ d^2f(x): x \in [0,1]^2 \} \subset Q_{\delta}$ with posterior probability tending to 1, the first term is bounded by a constant multiple of $| f^{(3,0)}(x)- f^{*(3,0)}(x)  |$, in view of the result (4) of Remark \ref{misc:integral_curves}. By the same remark and assumption $\| f^*\|_{\alpha, \infty}<\infty$, the second term is bounded by
	$ \|d^2f(x)- d^2f^*(x)\|| f^{*(3,0)}(x)| $. Using similar arguments for the second and third summand,
	one can see that the absolute value of the (1,1)th element of $\nabla V(x)-\nabla V^*(x)$ is bounded by $| f^{(3,0)}(x)- f^{*(3,0)}(x)  |+| f^{(2,1)}(x)- f^{*(2,1)}(x)  |+| f^{(1,2)}(x)- f^{*(1,2)}(x)  | $. Dealing the remaining elements of $\nabla V(x)-\nabla V^*(x)$ similarly, we have that  $\|\nabla V(x)-\nabla V^*(x)\|_F \lesssim \sum_{r: |r|=3} |D^{r}f(x)-D^{r}f^*(x)|$.	
\end{proof}\\

In above lemma, the optimal rates are obtained when $J\asymp (n/\log n)^{1/2(\alpha+1)}$, which then yields
\begin{align*}
	&\delta_{n,k,J}=\epsilon_n \asymp (\log n /n)^{(\alpha-k)/(2\alpha+2)}.
\end{align*}

In addition, we have the following two lemmas whose proofs closely  follow that of Lemma \ref{posterior_rates_f} and thus are omitted. Denote the posterior mean of $f$ by $\tilde{f}:=A_{(0,0)} Y+C_{(0,0)} \theta_0$ and similarly define the quantities it induces, for instance, $\tilde{V}$,$ {\tilde{\Upsilon}_{x_0}}$ and $\tilde{\mathcal{L}}$.  We then have the following lemma.
\begin{lemma} (Posterior contraction around the posterior mean) \label{posterior_rates_fmean}
	Under the above assumptions and $J_1=J_2=J$ and $J^2 \leq n$, with $J$ chosen such that $\eta_{n,k,J}$ vanishes to zero, then for any $M_n\to \infty $,
	\begin{align*}
		&\Pi(\sup_{x \in [0,1]^2} | D^{r}f(x)-D^{r}\tilde{f}(x)|>M_n\eta_{n,|r|,J}|\mathbb{D}_n)\xrightarrow{\mathrm{P}_0}0 , \\
		&\Pi( \sup_{x \in [0,1]^2}\| \nabla f(x)-\nabla \tilde{f}(x)\|>M_n\eta_{n,1,J}|\mathbb{D}_n)\xrightarrow{\mathrm{P}_0}0 ,\\
		&\Pi( \sup_{x \in [0,1]^2} \| d^2f(x)-d^2\tilde{f}(x)\|>M_n\eta_{n,2,J}|\mathbb{D}_n)\xrightarrow{\mathrm{P}_0}0, \\
		&\Pi( \sup_{x \in [0,1]^2} \|Hf(x)-H\tilde{f}(x)\|_F>M_n\eta_{n,2,J}|\mathbb{D}_n)\xrightarrow{\mathrm{P}_0}0, \\
		&\Pi( \sup_{x \in [0,1]^2} \|  \nabla d^2f(x)- \nabla d^2\tilde{f}(x)\|_F>M_n\eta_{n,3,J}|\mathbb{D}_n)\xrightarrow{\mathrm{P}_0}0, \\
		&\Pi( \sup_{x \in [0,1]^2} \|V(x)- \tilde{V}(x)\|_F>M_n\eta_{n,2,J}|\mathbb{D}_n)\xrightarrow{\mathrm{P}_0}0 ,\\
		&\Pi( \sup_{x \in [0,1]^2} \| \nabla V(x)-\nabla \tilde{V}(x) \|_F>M_n\eta_{n,3,J}|\mathbb{D}_n)\xrightarrow{\mathrm{P}_0}0,
	\end{align*}
	where $\eta_{n,k,J}= J^{k+1} (\log n/ n)^{1/2}$.	
\end{lemma}

One can also readily obtain the following convergence rates for the Bayesian estimators induced by $\tilde{f}$.

\begin{lemma} (Convergence of posterior mean) \label{convergence_rates_f}
	Under the above assumptions and $J_1  =J_2 =J $ and $J^2 \leq n$, with $J$ chosen such that $\delta_{n,k,J}$ vanishes to zero, then we have
	\begin{align*}
		&\sup_{x \in [0,1]^2} | D^{r}\tilde{f}(x)-D^{r}f^*(x)|=O_p( \delta_{n,|r|,J} ), \\
		& \sup_{x \in [0,1]^2}\| \nabla \tilde{f}(x)-\nabla f^*(x)\|=O_p(   \delta_{n,1,J} ), \\
		& \sup_{x \in [0,1]^2} \| d^2 \tilde{f}(x)-d^2f^*(x)\| =O_p( \delta_{n,2,J}) , \\
		&\sup_{x \in [0,1]^2} \|H\tilde{f}(x)-Hf^*(x)\|_F  =O_p( \delta_{n,2,J}), \\
		& \sup_{x \in [0,1]^2} \|  \nabla d^2 \tilde{f}(x)- \nabla d^2f^*(x)\|_F =O_p(  \delta_{n,3,J}), \\
		& \sup_{x \in [0,1]^2} \|\tilde{V}(x)- V^*(x)\|_F   =O_p(\delta_{n,2,J}) ,\\
		& \sup_{x \in [0,1]^2} \| \nabla \tilde{V}(x)-\nabla V^*(x) \|_F   =O_p( \delta_{n,3,J}),
	\end{align*}
	where $\delta_{n,k,J}= J^{k}\bigl( (\log n/ n) J^2+J^{-2\alpha} \bigr)^{1/2}$.	
\end{lemma}

We will also need the following lemma in the proof of Theorem \ref{rates:integral_curves}.
\begin{lemma} \label{bounds1}
	Let $ \tilde{G}(\intgx^*(s))=\frac{\partial G_1}{\partial x_1}(d^2f^*(\intgx^*(s)))$. For any $x_0, \tilde{x}_0 \in \mathcal{G}$, let $T^*_{x_0}=t^*_{\xz}+a^*$, $T^*_{\tilde{x}_0}=t^*_{\tilde{x}_0}+a^*$, $t \in [0, T^*_{x_0}]$ and $\tilde{t} \in [0, T^*_{\tilde{x}_0}]$. Define $a(\xz,t):=\int_0^t \tilde{G}(\intgx^*(s)) {\bjqr}(\intgx^* (s)) ds $ for $r=(2,0), (1,1)$ or $(0,2)$, where the integral is taken elementwise. Under Assumption (A1), (A2) and (A5) and $J_1=J_2=J$, we have
	\begin{align*}
	\|a(x_0,t)\|_1 \lesssim J^2, \quad	\|a(x_0,t)\|^2 \lesssim J^3,  \quad \|a(x_0,t)-a(\tilde{x}_0,\tilde{t}) \|^2 \lesssim J^3|t-\tilde{t}|+J^5\|x_0-\tilde{x}_0\|^2 .
	\end{align*}
\end{lemma}

\subsection{Proofs of the main results}

Since $\intgx(-t)=\int_{0}^t(-V(\intgx(-s)))ds$, with negative time it can be interpreted as a curve tracing in the reverse direction, i.e, $-V$. Since the direction of $V$ does not play a role in the theoretical proof, without loss of generality, $t$ is restricted on the $[0,T^*_{x_0}]$ where $T^*_{x_0}=t^*_{x_0}+a^*$ and we shall assume the hitting times are nonnegative.

\begin{proof}[Proof of Theorem \ref{rates:integral_curves}]
	We first sketch a proof for the following result as first proved in \citet{koltchinskii2007annals}:
	\[
	\sup_{x_0 \in \mathcal{G}}\sup_{t \in [0,T^*_{x_0}]}\left\| \intgx(t)-\intgx^*(t) \right\| \lesssim \sup_{x_0 \in \mathcal{G}}\sup_{t \in [0,T^*_{x_0}]} \left\| \int_0^t  (V-V^*)(\intgx^*(s))ds \right\|.
	\]
	To see this, let
	\begin{align*}
		y_{\xz}(t)&:=\intgx(t)-\intgx^*(t)=\int_{0}^t \left( V(\intgx(s))-V^*(\intgx^*(s) \right)ds,  \\
		z_{\xz}(t)&:=\int_{0}^t  \left( V-V^* \right)\intgx^*(s)ds+ \int_{0}^t \nabla V^*(\intgx^*(s))z_{\xz}(s)ds.
	\end{align*}
	Note that $\delta_{\xz}(t):=y_{\xz}(t)-z_{\xz}(t)= \int_{0}^t \nabla V^*(\intgx^*(s))\delta(s)ds  +R_{\xz}(t)$ for some reminder term $R_{\xz}(t)$.
	So $\| \delta_{\xz}(t) \| \leq \| R_{\xz}(t) \|+ \int_{0}^t \| \nabla V^*(\intgx^*(s)) \|_F \| \delta_{\xz}(s) \|ds$, by the Gronwall-Bellman inequality \citep{kim2009gronwall},
	\[
	\| \delta_{\xz}(t) \| \leq \sup_{t \in [0,T^*_{\xz}]} \left( \| R_{\xz}(t)\| \exp \Big(\int_{0}^t \| \nabla V^*(\intgx^*(s)) \|_F ds \Big)  \right).
	\]
	Hence $ \sup_{t \in [0,T^*_{\xz}]}\| \delta_{\xz}(t) \| )\lesssim \sup_{t\in [0,T^*_{\xz}]} \| R_{\xz}(t)\| $.
	It can be showed that with high posterior probability $\sup_{t\in [0,T^*_{\xz}]} \| R_{\xz}(t)\| \ll  \int_{0}^{T^*_{\xz}} \|y_{\xz}(s) \|ds   $ in $\mathrm{P}_0$-probability tending to $1$ following the argument in page 1586 of \citet{koltchinskii2007annals}. Therefore,  $\sup_{t \in [0,T^*_{\xz}]}\| \delta_{\xz}(t) \| \ll \int_{0}^{T^*_{\xz}} \|y_{\xz}(s) \| ds $. Since $y_{\xz}(t)=\delta_{\xz}(t)+z_{\xz}(t)$, it follows that $\sup_{t \in [0,T^*_{\xz}]}\| \delta_{\xz}(t) \| \ll \int_{0}^{T^*_{\xz}} \|z_{\xz}(s) \| ds \lesssim \sup_{t \in [0,T^*_{\xz}]}\| z_{\xz}(t) \| $. Then $ \sup_{t \in [0,T^*_{\xz}]}\| y_{\xz}(t) \| =\sup_{t \in [0,T^*_{\xz}]}\| \delta_{\xz}(t)+ z_{\xz}(t) \|\lesssim \sup_{t \in [0,T^*_{\xz}]}\| z_{\xz}(t) \| $.
But
\[
\| z_{\xz}(t) \|  \leq \left\| \int_{0}^t   \left( V-V^* \right)\intgx^*(s)ds  \right\|+ \int_{0}^t \| \nabla V^*(\intgx^*(s)) \|_F \|z_{\xz}(s)\|ds.
\]
One more application of the Gronwall-Bellman inequality and then taking the supremum on the left hand side yields
	\begin{align*}
		\sup_{t \in [0,T^*_{\xz}]}\| z_{\xz}(t) \| &\leq \sup_{t \in [0,T^*_{\xz}]} \left( \left\| \int_{0}^t  \left( V-V^* \right)\intgx^*(s)ds \right\| \exp \Big(\int_{0}^t \| \nabla V^*(\intgx^*(s)) \|_F ds \Big)  \right) \\
		&\lesssim \sup_{t\in [0,T^*_{\xz}]} \left\| \int_{0}^t  \left( V-V^* \right)\intgx^*(s)ds \right\|.
	\end{align*}
	Taking another supremum over all $\xz$ gives us the result.
	
	Now coming back to the main proof,
	by  an integral form Taylor expansion of $V(x)(=G(d^2f(x)))$ around $G(d^2f^*(x))$, by the uniform boundedness of each element of second derivative $G$, one can get
	\[
	\sup_{\substack{x_0 \in \mathcal{G} \\ t\in[0,T^*_{x_0}]}} \left\| \int_0^t \bigl[ (V-V^*)(\intgx^*(s))-
	\nabla G(d^2f^*(\intgx^* (s))) d^2(f-f^*)(\intgx^* (s))\bigr] ds
	\right\|\lesssim \sup_{\xz}\|d^2f(x)-d^2f^*(x) \|^2.
	\]
	In view of Lemma \ref{posterior_rates_f}, with the choice of $J$ in the assumption,  $\sup_{x}\|d^2f(x)-d^2f^*(x) \|^2$ is of order $(n/ \log n)^{(5-\alpha)/(1+2\alpha)}$.
	Now it suffices to prove that
	\[
	\sup_{\substack{x_0 \in \mathcal{G} \\ t\in[0,T^*_{x_0}]}} \left\| \int_0^t \nabla G(d^2f^*(\intgx^* (s))) d^2(f-f^*)(\intgx^* (s)) ds \right\|
	\]
	has posterior contraction rate  $(n/\log n)^{(2-\alpha)/(1+2\alpha)}$.
	To this end, we demonstrate the following quantity has this posterior contraction rate
	\[
	\sup_{\substack{x_0 \in \mathcal{G} \\ t\in[0,T^*_{x_0}]}} \left| \int_0^t \frac{\partial G_1}{\partial {x_1}}(d^2f^*(\intgx^* (s))) (\ftz-{\ftz}^*)(\intgx^* (s)) ds \right|.
	\]
	The rate for the other components can be derived similarly. Let $\mathcal{U}_n $ be a shrinking neighborhood of $\sigma^2_0$ such that $\hat{\sigma}^2_n \in \mathcal{U}_n $ and $\Pi(\sigma^2 \in \mathcal{U}_n| \mathbb{D}_n ) \to 1$ with probability tending to 1. \\
	Let $\tilde{G}(\intgx^*(s))= \frac{\partial G_1}{\partial {x_1}}(d^2f^*(\intgx^* (s)))$. To derive the rate, it suffices to show the rate for
	\begin{equation}
		{\rm E}_0\Bigl[\sup_{\sigma^2 \in \mathcal{U}_n} {\rm E}\Bigl( \sup_{x_0,t}\Bigl| \int_0^t \tilde{G}(\intgx^*(s))  (\ftz-{\ftz}^*)(\intgx^* (s)) ds  \Bigr|^2 \Bigr| \mathbb{D}_n,\sigma^2  \Bigr)\Bigr]  \label{integral:eqn:1}
	\end{equation}
	decays as $(n/\log n)^{(4-2\alpha)/(1+2\alpha)}$,
	which then by Markov's inequality yields the desired posterior contraction rate. Notice that the posterior variance of $\ftz(x)$ does not depend on $\mathbb{D}_n$, while its posterior mean $A_{(2,0)}Y+C_{(2,0)}\theta_0$ does not depend on $\sigma^2$. Also, in the posterior distribution conditional on $\sigma^2$, we have $\ftz(x)-A_{(2,0)}(x)Y-C_{(2,0)}(x)\theta_0=\sigma {b^{(2,0)}_{J_1,J_2}}^T(x) (B^TB+\Lambda_0^{-1} )^{-1/2} Z, $ where $Z$ is a $J_1J_2 \times 1$ vector of standard normal random variables. Thus we can write
	\begin{align}
		&{\rm E}_0\Bigl[\sup_{\sigma^2 \in \mathcal{U}_n} {\rm E}\Bigl( \sup_{x_0,t}\Bigl| \int_0^t \tilde{G}(\intgx^*(s))  (\ftz-{\ftz}^*)(\intgx^* (s)) ds  \Bigr|^2 \Bigr| \mathbb{D}_n,\sigma^2  \Bigr)\Bigr] \notag \\
		&\lesssim {\rm E}_0\Bigl[\sup_{\sigma^2 \in \mathcal{U}_n} {\rm E}\Bigl( \sup_{x_0,t}\Bigl| \int_0^t \tilde{G}(\intgx^*(s)) (\ftz-A_{(2,0)}Y-C_{(2,0)}\theta_0)(\intgx^* (s)) ds  \Bigr|^2 \Bigr| \sigma^2  \Bigr)\Bigr] \notag  \\
		&+  {\rm E}_0\Bigl[ \sup_{x_0,t}\Bigl| \int_0^t \tilde{G}(\intgx^*(s)) (A_{(2,0)}Y+C_{(2,0)}\theta_0-{\ftz}^*)(\intgx^* (s)) ds  \Bigr|^2   \Bigr] \notag \\
		&\lesssim  \sup_{\sigma^2 \in \mathcal{U}_n} \sigma^2 {\rm E}_0\Bigl( \sup_{x_0,t}\Bigl| \int_0^t \tilde{G}(\intgx^*(s)) {b^{(2,0)}_{J_1,J_2}}^T(\intgx^* (s)) ds (B^TB+\Lambda_0^{-1} )^{-1/2}  Z \Bigr|^2   \Bigr) \label{integral:eqn:2}\\
		&+ \sup_{x_0,t}\left| \int_0^t \tilde{G}(\intgx^*(s)) (A_{(2,0)}(\intgx^* (s))F^*+C_{(2,0)}(\intgx^* (s))\theta_0-{\ftz}^*(\intgx^* (s))) ds \right|^2 \label{integral:eqn:3} \\
		&+    {\rm E}_0\left(\sup_{x_0,t} \Bigl|   \int_0^t \tilde{G}(\intgx^*(s))A_{(2,0)}(\intgx^* (s))\varepsilon ds  \Bigr|^2\right) \label{integral:eqn:4}.
	\end{align}
	
	We shall tackle each of above three terms one by one. Let $a(\xz,t)=\int_0^t \tilde{G}(\intgx^*(s)) {b^{(2,0)}_{J_1,J_2}}(\intgx^* (s)) ds $, where the integral is taken elementwise.
	Now consider the first term \eqref{integral:eqn:2}. Define $U_{1,n}(\xz,t)=a(\xz,t)(B^TB+\Lambda_0^{-1} )^{-1/2}  Z$ which is Gaussian. For fixed $x_0, t$,
	\begin{align*}
		{\rm E} [U_{1,n}(x_0,t)]^2=a(\xz,t)^T(B^TB+\Lambda_0^{-1}  )a(\xz,t) \lesssim \frac{J^2}{n} \|a(\xz,t)\|^2 \lesssim \frac{J^5}{n},
	\end{align*}
	where last step follows from Lemma \ref{bounds1}.
	Consider the two points $(x_0, t)$ and $(\tilde{x}_0,\tilde{t})$,	where the distance
	\begin{align*}
		d((x_0, t),(\tilde{x}_0,\tilde{t}))&:=\sqrt{ {\rm Var}(U_{1,n}(x_0, t)-{U_{1,n}}(\tilde{x}_0,\tilde{t}))}\\
		&\lesssim \Bigl( (a(x_0,t)-a(\tilde{x}_0,\tilde{t}))^T (B^TB+\Lambda_0^{-1} )(a(x_0,t)-a(\tilde{x}_0,\tilde{t}))   \Bigr)^{1/2} \\
		& \lesssim  \Bigl( \frac{J^2}{n} \Bigr)^{1/2}\|a(x_0,t)-a(\tilde{x}_0,\tilde{t}) \|   \\
		&\lesssim  \Bigl( \frac{J^7}{n}(|t-\tilde{t}|+\| x_0-\tilde{x}_0 \|^2) \Bigr)^{1/2};
	\end{align*}
	here the last line follows from Lemma \ref{bounds1}. Let $\tilde{d}( (x_0, t),(\tilde{x}_0,\tilde{t})  ):=(|t-\tilde{t}|+\| x_0-\tilde{x}_0 \|^2)^{1/2} $. Thus $d((x_0, t),(\tilde{x}_0,\tilde{t}))\lesssim \rho_n \tilde{d}( (x_0, t),(\tilde{x}_0,\tilde{t})) $ for $\rho_n =J^{7/2}/n^{1/2}$. By Lemma A.11 of \citet{yoo2016supremum}, setting $\delta_n \asymp 1/J$ there, we obtain
	${\rm E}\Big(\sup_{x_0,t}[ U_{1,n}(x_0, t)]^2\Bigr) =O(J^5 \log n /n )$.
	Therefore, \eqref{integral:eqn:2} is of order $(\log n) J^5/n$.
	
	For the term \eqref{integral:eqn:4}, let $U_{2,n}(x_0, t)= \int_0^t \tilde{G}(\intgx^*(s))A_{(2,0)}(\intgx^* (s))\varepsilon ds=a(\xz,t)(B^TB+\Lambda_0^{-1} )^{-1} B^T \varepsilon$. Observe that for fix $x_0, t$,
	\begin{align*}
		{\rm E}[U_{2,n}(x_0, t)]^2&=\sigma_0^2 a(x_0,t)^T(B^TB+\Lambda_0^{-1} )^{-1} B^TB (B^TB+\Lambda_0^{-1} )^{-1} a(x_0,t) \\
		&\lesssim \sigma_0^2 \frac{J^2}{n} \frac{n}{J^2} \frac{J^2}{n} a(x_0,t)^Ta(x_0,t) \\
		&\lesssim \sigma_0^2 \frac{J^5}{n}.
	\end{align*}
	Since $\varepsilon$ is sub-Gaussian, $U_{2,n}(x_0,t)$ is sub-Gaussian with mean $0$ and variance ${\rm E}[U_{2,n}(x_0, t)]^2$. By the same argument,
	$ {\rm E}\Big(\sup_{x_0,t}[ U_{2,n}(x_0, t)]^2\Bigr) =O(J^5 \log n /n )$. Therefore, \eqref{integral:eqn:4} is of order $(\log n) J^5/n$.
	
	Finally, for the term \eqref{integral:eqn:3}, in view of Lemma \ref{approx_error}, write
	\begin{align*}
		&\left| \int_0^t \tilde{G}(\intgx^*(s)) \Bigl(A_{(2,0)}(\intgx^* (s))F^*+C_{(2,0)}(\intgx^* (s))\theta_0-{\ftz}^*(\intgx^* (s))\Bigr) ds \right| \\
		=&\left| \int_0^t \tilde{G}(\intgx^*(s)){b^{(2,0)}_{J_1,J_2}}^T(\intgx^* (s)) (B^TB+\Lambda_0^{-1} )^{-1} \Bigl(B^T (F^*-B\theta_\infty)+\Lambda_0^{-1}(\theta_0-\theta_\infty) \Bigr)   ds  \right| \\
		+& \left| \int_0^t \tilde{G}(\intgx^*(s))\Bigl({b^{(2,0)}_{J_1,J_2}}^T(\intgx^* (s))  \theta_\infty- {\ftz}^*(\intgx^* (s)) \Bigr) ds \right|.
	\end{align*}
	The first term on the right hand side of above equation is just
	\[
	\left| a(x_0,t)^T (B^TB+\Lambda_0^{-1} )^{-1} \Bigl(B^T (F^*-B\theta_\infty)+\Lambda_0^{-1}(\theta_0-\theta_\infty) \Bigr) \right|,
	\]
	which is  bounded by
	\[
	\|a(x_0,t)\|_1 \| (B^TB+\Lambda_0^{-1} )^{-1}\|_{(\infty,\infty)} \Bigl( \| B^T (F^*-B\theta_\infty)\|_{\infty}+ \| \Lambda_0^{-1}\|_{(\infty,\infty)}(\|\theta_0 \|_{\infty}+\|\theta_{\infty} \|_{\infty})\Bigr) .
	\]
	Note that $\|a(x_0,t)\|_1 \lesssim J^2$ by Lemma \ref{bounds1}. Since $\|(B^TB+\Lambda_0^{-1} )^{-1}\|_{(\infty,\infty)}=O(J^2/n)$ and $\| B^T (F^*-B\theta_\infty)\|_{\infty}=O(nJ^{-2-\alpha})$, this term is of order $(J^4/n)+J^{2-\alpha}$.  The second term of the right hand side of above equation can be bounded by the supremum of approximate error which is of order $J^{2-\alpha}$. Therefore, \eqref{integral:eqn:3} is of order $J^4((J^4/n^2)+J^{-2\alpha})$.
	
	Putting all the above terms together, \eqref{integral:eqn:1} is of order $(\log n/n) J^5+ J^4((J^4/n^2)+J^{-2\alpha}) $. For $J \asymp (n/\log n)^{1/(1+2\alpha)}$, \eqref{integral:eqn:1} is of order $(n/\log n)^{(4-2\alpha)/(1+2\alpha)}$, which then completes the proof.
\end{proof}\\

\begin{proof}[Proof of Lemma \ref{bounds1}]
	We shall consider only two separate cases (i) $r=(2,0)$ and (ii) $r=(1,1)$.\\
	(i). For $r=(2,0)$ (similarly for $r=(0,2)$),
	\begin{align*}
		\|a(x_0,t)\|_1&=\doublesum \Bigl|\int_0^t \tilde{G}_1(s) B''_{j_1}(\intgxo^*(s))B_{j_2}(\intgxt^*(s))ds \Bigr| \\
		& \lesssim \doublesum \int_0^t|B''_{j_1}(\intgxo^*(s)) |B_{j_2}(\intgxt^*(s))ds \\
		& \leq \sum_{j_1}^{J_1} \int_0^t |B''_{j_1}(\intgxo^*(s)) |   ds,
	\end{align*}
	the last line follows by $\sum_{j_2=1}^{J_2}B_{j_2}(x_2)=1$ and using the fact that
	\begin{align*}
		B''_{j_1,q_1}(x)&=\frac{(q_1-1)(q_1-2)B_{j_1,q_1-2}(x)}{(t_{j_1}-t_{j_1-q_1+2})(t_{j_1}-t_{j_1-q_1+1})}
		+\frac{(q_1-1)(q_1-2)B_{j_1-1,q_1-2}(x)}{(t_{j_1-1}-t_{j_1-q_1+1})(t_{j_1}-t_{j_1-q_1+1})} \\
		&- \frac{(q_1-1)(q_1-2)B_{j_1-1,q_1-2}(x)}{(t_{j_1-1}-t_{j_1-1-q_1+2})(t_{j_1-1}-t_{j_1-q_1})}
		+\frac{(q_1-1)(q_1-2)B_{j_1-2,q_1-2}(x)}{(t_{j_1-2}-t_{j_1-1-q_1+1})(t_{j_1-1}-t_{j_1-q_1})},
	\end{align*}
	which implies $\sum_{j_1=1}^{J_1}| B''_{j_1}(\intgxo^*(s))  | \lesssim 4J^2 $. Therefore, $\|a(x_0,t)\|_1 \lesssim J^2$.
	
	Next, let $S_{j_1}=[t_{1,j_1-(q_1-2)},t_{1,j_1}]$, $S_{j_2}=[t_{2,j_2-q_2},t_{2,j_2}]$ and  $1_{j_1,j_2}(s):=\mathbbm{1}\{s: \intgx^* (s) \in S_{j_1} \times S_{j_2}\}$. Turn to $\|a(x_0,t)\|^2$, which is
	\begin{align*}
		& \quad \quad \doublesum \Bigl(\int_0^t \tilde{G}_1(s) B''_{j_1}(\intgxo^*(s))B_{j_2}(\intgxt^*(s))ds \Bigr)^2 \\
		&=\int_0^t \int_0^t \tilde{G}_1(s) \tilde{G}_1(s') \doublesum \Bigl( B''_{j_1}(\intgxo^*(s))B_{j_2}(\intgxt^*(s))B''_{j_1}(\intgxo^*(s'))B_{j_2}(\intgxt^*(s')) \Bigr)    dsds' \\
		&\lesssim \int_0^t \int_0^t  \doublesum \Bigl( |B''_{j_1}(\intgxo^*(s))|B_{j_2}(\intgxt^*(s))|B''_{j_1}(\intgxo^*(s'))|B_{j_2}(\intgxt^*(s')) \Bigr)    dsds' \\
		&= \int_0^t \int_0^t  \doublesum 1_{j_1,j_2}(s) 1_{j_1,j_2}(s') \Bigl( |B''_{j_1}(\intgxo^*(s))|B_{j_2}(\intgxt^*(s))|B''_{j_1}(\intgxo^*(s'))|B_{j_2}(\intgxt^*(s')) \Bigr)    dsds'.
	\end{align*}
	The last equality is obtained as follows. Since for any fix $j_1, j_2$, $B''_{j_1}(\cdot)$ is supported on $S_{j_1}$ and $B_{j_2}(\cdot)$ is supported on $S_{j_2}$. So $1_{j_1,j_2}(s) 1_{j_1,j_2}(s')=\mathbbm{1}\{(s,s'):\intgx^* (s) \in S_{j_1} \times S_{j_2} \text{ and } \intgx^* (s') \in S_{j_1} \times S_{j_2}  \}$.
	Notice that for $n$ large, $|S_{j_1}| \asymp |S_{j_2}| \asymp J^{-1}$ and $\intgx^* (s) \in S_{j_1} \times S_{j_2} \text{ and } \intgx^* (s') \in S_{j_1} \times S_{j_2} $ implies that $\|\intgx^* (s)- \intgx^* (s')  \|\lesssim J^{-1}$. By Assumption (A5), $C_{\mathcal{G}}|s-s'|\leq |s-s'\|(\intgx^*(s)-\intgx^*(s'))/(s-s')\|\lesssim J^{-1}$, and hence $|s-s'|\lesssim J^{-1}$. Therefore, for $n$ large enough, above quantity can be further bounded by a constant multiple of
\[
\int_0^t \int_0^t  \doublesum \mathbbm{1}\{|s-s'|<CJ^{-1}\} \Bigl(|B''_{j_1}(\intgxo^*(s))|B_{j_2}(\intgxt^*(s))|B''_{j_1}(\intgxo^*(s'))|B_{j_2}(\intgxt^*(s')) \Bigr)    dsds'.
\]
Noting $B_{j_2}(x_2)\leq 1$ and $\sum_{j_2=1}^{J_2}B_{j_2}=1$, it can be further bounded by
	\begin{align*}
		& \int_0^t \int_0^t  \doublesum \mathbbm{1}\{|s-s'|<CJ^{-1}\} \Bigl(|B''_{j_1}(\intgxo^*(s))|B_{j_2}(\intgxt^*(s))|B''_{j_1}(\intgxo^*(s'))| \Bigr)    dsds' \\
		& \lesssim \int_0^t \int_0^t  \sum_{j_1}^{J_1} \mathbbm{1}\{|s-s'|<CJ^{-1}\} \Bigl(|B''_{j_1}(\intgxo^*(s))B''_{j_1}(\intgxo^*(s'))| \Bigr)    dsds' \\
		& \lesssim J^2 J^{-1} \int_0^t \sum_{j_1=1}^{J_1}  | B''_{j_1}(\intgxo^*(s))|ds.
	\end{align*}
	From argument used in bounding $\|a(\xz, t) \|_1$, we have $\sum_{j_1=1}^{J_1}| B''_{j_1}(\intgxo^*(s))  | \lesssim J^2 $. This completes the proof for $\|a(x_0,t)\|^2 \lesssim J^3 $. \\
	
	For the third result, we write
	\[
	\|a(\xz,t)-a(\tilde{x}_0,\tilde{t})\|^2 \lesssim \|a_1(x_0,t,\tilde{t}) \|^2+\|a_2(\tilde{t},x_0,\tilde{x}_0)  \|^2,
	\]
	where
	\[
	a_1(x_0,t,\tilde{t}) :=\int_0^t \tilde{G}(\intgx^*(s)) { b^{(2,0)}_{J_1,J_2}  }(\intgx^* (s)) ds-\int_0^{\tilde{t}} \tilde{G}(\intgx^*(s)) { b^{(2,0)}_{J_1,J_2}  }(\intgx^* (s)) ds,
	\]
	and
	\[
	a_2(\tilde{t},x_0,\tilde{x}_0):=\int_0^{\tilde{t}} \tilde{G}(\intgx^*(s)) { b^{(2,0)}_{J_1,J_2}  }(\intgx^* (s)) ds-\int_0^{\tilde{t}} \tilde{G}( \Upsilon_{\tilde{x}_0}  ^*(s)) { b^{(2,0)}_{J_1,J_2}  }( \Upsilon_{\tilde{x}_0}   ^* (s)) ds.
	\]
	First, note that
	\begin{align*}
		\| a_1(x_0,t,\tilde{t}) \|^2&=\doublesum \left( \int_{\tilde{t}}^t \tilde{G}(\intgx^*(s)) B''_{j_1}(\intgxo^*(s))B_{j_2}(\intgxt^*(s)) ds  \right)^2 \\
		&\lesssim J^2 J^{-1}\int_{t}^{\tilde{t}}(4J^2)ds \\
		&=J^3|t-\tilde{t}|,
	\end{align*}
	where the second line follows by a similar argument used to bound $\|a(x_0,t)\|^2$.
	Next, $\| a_2(\tilde{t},x_0,\tilde{x}_0) \|^2 $ is given by
	\[
	\doublesum \Bigl[ \int_0^{\tilde{t}} \Bigl(\tilde{G}(\intgx^*(s))B''_{j_1}(\intgxo^*(s))B_{j_2}(\intgxt^*(s))- \tilde{G}(\Upsilon^*_{\tilde{x}_0}(s))B''_{j_1}(\Upsilon^*_{1,\tilde{x}_0}(s) )B_{j_2}(\Upsilon^*_{2,\tilde{x}_0} (s) )ds \Bigr)     \Bigr]^2 ,
	\]
	which is bounded (up to a constant multiple) by
	\begin{align*}
		&\doublesum \Bigl(  \int_0^{\tilde{t}} \tilde{G}(\intgx^*(s))\Bigl( B''_{j_1}(\intgxo^*(s))-  B''_{j_1}(\Upsilon^*_{1,\tilde{x}_0}(s) )\Bigr) B_{j_2}(\intgxt^*(s))  ds   \Bigr)^2 \\
		&+\doublesum \Bigl(   \int_0^{\tilde{t}}  \tilde{G}(\intgx^*(s))  \Bigl( B_{j_2}(\intgxt^*(s))-  B_{j_2}(\Upsilon^*_{2,\tilde{x}_0}(s) )\Bigr) B''_{j_1}(\Upsilon^*_{1,\tilde{x}_0}(s) ) ds     \Bigr)^2\\
		&+ \doublesum \Bigl(   \int_0^{\tilde{t}}  \Bigl( \tilde{G}(\intgx^*(s))- \tilde{G}(\Upsilon^*_{\tilde{x}_0}(s) ) \Bigr)  B''_{j_1}(\Upsilon^*_{1,\tilde{x}_0}(s))  B_{j_2}(\Upsilon^*_{2,\tilde{x}_0}(s))  ds     \Bigr)^2.
	\end{align*}
	Bound the first term in the right hand side of above expression as
	\begin{align*}
		&\doublesum \Bigl(  \int_0^{\tilde{t}} \tilde{G}(\intgx^*(s))\Bigl(  B''_{j_1}(\intgxo^*(s))-  B''_{j_1}(\Upsilon^*_{1,\tilde{x}_0}(s) ) \Bigr) B_{j_2}(\intgxt^*(s))  ds   \Bigr)^2 \\
		\lesssim &  \doublesum \Bigl(  \int_0^{\tilde{t}} |\tilde{G}(\intgx^*(s))|   |B'''_{j_1}(\intgxo^*(s))| |   \intgxo^*(s)-  \Upsilon^*_{1,\tilde{x}_0}(s)   | B_{j_2}(\intgxt^*(s))  ds   \Bigr)^2   \\	
		\lesssim &  \|x_0-\tilde{x}_0\|^2 \doublesum \Bigl(  \int_0^{\tilde{t}} |\tilde{G}(\intgx^*(s))|   |B'''_{j_1}(\intgxo^*(s))| B_{j_2}(\intgxt^*(s))  ds   \Bigr)^2  \\
		\lesssim & \|x_0-\tilde{x}_0\|^2 J^2    \int_{0}^{\tilde{t}} |B'''_{j_1}(\intgxo^*(s))| ds  \\
		 \lesssim &  \|x_0-\tilde{x}_0\|^2  J^5.
	\end{align*}
	The second term is bounded as
	\begin{align*}
		&\doublesum \Bigl(   \int_0^{\tilde{t}}  \tilde{G}(\intgx^*(s))  \Bigl( B_{j_2}(\intgxt^*(s))-  B_{j_2}(\Upsilon^*_{2,\tilde{x}_0}(s) )\Bigr) B''_{j_1}(\Upsilon^*_{1,\tilde{x}_0}(s) ) ds     \Bigr)^2\\
		\lesssim &  \doublesum  \Bigl(   \int_{0}^{\tilde{t}} |\tilde{G}(\intgx^*(s)) | | \intgxt^*(s)-  \Upsilon^*_{2,\tilde{x}_0}(s) |
		|B'_{j_2}( \Upsilon^*_{2,\tilde{x}_0}(s)   )| |B''_{j_1}(\Upsilon^*_{1,\tilde{x}_0}(s))|    \Bigr)^2  \\		
		\lesssim & \| x-\tilde{x}_0 \|^2 \doublesum  \Bigl(   \int_{0}^{\tilde{t}} |\tilde{G}(\intgx^*(s)) |
		 |B'_{j_2}( \Upsilon^*_{2,\tilde{x}_0}(s)   )| |B''_{j_1}(\Upsilon^*_{1,\tilde{x}_0}(s))|    \Bigr)^2  \\
		\lesssim & \| x-\tilde{x}_0 \|^2  \int_{0}^{\tilde{t}} \mathbbm{1}\{|s-s'|\leq CJ^{-1} \}   \\
            & \times \doublesum |B'_{j_2}( \Upsilon^*_{2,\tilde{x}_0}(s))||B''_{j_1}(\Upsilon^*_{1,\tilde{x}_0}(s))||B'_{j_2}( \Upsilon^*_{2,\tilde{x}_0}(s')   )| |B''_{j_1}(\Upsilon^*_{1,\tilde{x}_0}(s'))|     dsds' \\
		\lesssim & \| x-\tilde{x}_0 \|^2  \int_{0}^{\tilde{t}} \mathbbm{1}\{|s-s'|\leq CJ^{-1} \} J^2 \sum_{j_1=1}^{J_1}
		|B''_{j_1}(\Upsilon^*_{1,\tilde{x}_0}(s))| |B''_{j_1}(\Upsilon^*_{1,\tilde{x}_0}(s'))|     dsds' \\				
		 \lesssim & J^5 \|x_0-\tilde{x}_0\|^2,
	\end{align*}
	where the second line follows from the mean value theorem, the third line from the Lipschitz continuity of $\intgx^*$ in $x_0$ (Remark \ref{misc:integral_curves}) whereas fourth line follows by a similar argument used to bound $\|a(x,t)\|^2$.
	The third term
\begin{align*}
\doublesum & \Bigl(   \int_0^{\tilde{t}}   \Bigl( \tilde{G}(\intgx^*(s))- \tilde{G}(\Upsilon^*_{\tilde{x}_0}(s) ) \Bigr)  B''_{j_1}(\Upsilon^*_{1,\tilde{x}_0}(s))  B_{j_2}(\Upsilon^*_{2,\tilde{x}_0}(s))  ds     \Bigr)^2 \\
	\lesssim	& \doublesum   \int_0^{\tilde{t}} \Bigl( \| \intgx^*(s)-\Upsilon^*_{\tilde{x}_0}(s) \| B''_{j_1}(\Upsilon^*_{1,\tilde{x}_0}(s))  B_{j_2}(\Upsilon^*_{2,\tilde{x}_0}(s))  ds     \Bigr)^2 \\
		\lesssim & \| \xz -\tilde{x}_0  \|^2  \doublesum \Bigl(  \int_0^{\tilde{t}}  |B''_{j_1}(\Upsilon^*_{1,\tilde{x}_0}(s))|  B_{j_2}(\Upsilon^*_{2,\tilde{x}_0}(s))  ds     \Bigr)^2  \\
		 \lesssim & \| \xz -\tilde{x}_0  \|^2 J^3,
	\end{align*}		
	where the second line holds by mean value theorem and the third line holds due to Lipschitz continuity of $\intgx^*$ in $x_0$ (Remark \ref{misc:integral_curves}) and last line holds by similar argument for $\| a(\xz, t)\|^2$.
	
	In summary, we have that $\| a_2(\tilde{t},x_0,\tilde{x}_0) \|^2 \lesssim J^5 \|x_0-\tilde{x}_0\|^2$ and $\|a_1(x_0,t,\tilde{t})  \|^2\lesssim J^3|t-\tilde{t}|$.\\
	
\noindent	
	(ii). Now turning to the case $r=(1,1)$. By similar argument we have
	\begin{align*}
		\|a(x_0,t)\|_1&=\doublesum \Bigl|\int_0^t \tilde{G}_1(s) B'_{j_1}(\intgxo^*(s))B'_{j_2}(\intgxt^*(s))ds \Bigr|\\
		& \lesssim \doublesum \int_0^t|B'_{j_1}(\intgxo^*(s)) \|B'_{j_2}(\intgxt^*(s))|ds \\
		& \leq J \sum_{j_1=1}^{J_1} \int_0^t |B'_{j_1}(\intgxo^*(s)) |   ds \lesssim J^2
	\end{align*}
	Likewise, $ \|a(x_0,t)\|^2=\doublesum \Bigl(\int_0^t \tilde{G}_1(s) B'_{j_1}(\intgxo^*(s))B'_{j_2}(\intgxt^*(s))ds \Bigr)^2 $ can be bounded by
	\begin{align*}
		& \int_0^t \int_0^t  \doublesum \mathbbm{1}\{|s-s'|<CJ^{-1}\} \Bigl(|B'_{j_1}(\intgxo^*(s))B'_{j_2}(\intgxt^*(s))B'_{j_1}(\intgxo^*(s'))B'_{j_2}(\intgxt^*(s'))| \Bigr)    dsds' \\
		& \lesssim J \int_0^t \int_0^t  \doublesum \mathbbm{1}\{|s-s'|<CJ^{-1}\} \Bigl(|B'_{j_1}(\intgxo^*(s))\|B'_{j_2}(\intgxt^*(s))\|B'_{j_1}(\intgxo^*(s'))| \Bigr)    dsds'\\
		& \lesssim J^2 \int_0^t \int_0^t  \sum_{j_1=1}^{J_1} \mathbbm{1}\{|s-s'|<CJ^{-1}\} \Bigl(|B'_{j_1}(\intgxo^*(s))B'_{j_1}(\intgxo^*(s'))| \Bigr)    dsds' \\
		& \lesssim J^3 \int_0^t \int_0^t  \mathbbm{1}\{|s-s'|<CJ^{-1}\} \sum_{j_1}^{J_1} \Bigl(|B'_{j_1}(\intgxo^*(s))| \Bigr)    dsds' \\
		& \lesssim J^3 J^{-1} \int_0^t \sum_{j_1=1}^{J_1}  | B'_{j_1}(\intgxo^*(s))|ds \\
		& \lesssim J^3.
	\end{align*}
	The third result $\|a_1(x_0,t,\tilde{t})  \|^2 \lesssim J^3|t-\tilde{t}|$ and $\| a_2(\tilde{t},x,\tilde{x}_0)\|^2  \lesssim J^5\| \xz-\tilde{x}_0\|^2  $ can be derived in a similar manner, since
	\begin{align*}
		\| a_1(x_0,t,\tilde{t}) \|^2&=\doublesum \left( \int_{\tilde{t}}^t \tilde{G}(\intgx^*(s)) B'_{j_1}(\intgxo^*(s))B'_{j_2}(\intgxt^*(s)) ds  \right)^2 \\
		&\lesssim J^2 J^{-1}\int_{t}^{\tilde{t}}(4J^2)ds \\
		&=J^3|t-\tilde{t}|,
	\end{align*}
	where the second line follows by a similar argument used in bounding $\|a(x_0,t)\|^2$.
	Next, to bound $\| a_2(\tilde{t},x_0,\tilde{x}_0) \|^2 $, we need to estimate the following three terms
	\begin{align*}
		&\doublesum \Bigl(  \int_0^{\tilde{t}} \tilde{G}(\intgx^*(s))\Bigl(  B'_{j_1}(\intgxo^*(s))-  B'_{j_1}(\Upsilon^*_{1,\tilde{x}_0}(s) ) \Bigr) B'_{j_2}(\intgxt^*(s))  ds   \Bigr)^2, \\
		&\doublesum \Bigl(   \int_0^{\tilde{t}}  \tilde{G}(\intgx^*(s))  \Bigl( B'_{j_2}(\intgxt^*(s))-  B'_{j_2}(\Upsilon^*_{2,\tilde{x}_0}(s) )\Bigr) B'_{j_1}(\Upsilon^*_{1,\tilde{x}_0}(s) ) ds     \Bigr)^2,\\
   &\doublesum \Bigl(   \int_0^{\tilde{t}}  \Bigl( \tilde{G}(\intgx^*(s))- \tilde{G}(\Upsilon^*_{\tilde{x}_0}(s) ) \Bigr)  B'_{j_1}(\Upsilon^*_{1,\tilde{x}_0}(s))  B'_{j_2}(\Upsilon^*_{2,\tilde{x}_0}(s))  ds     \Bigr)^2 	.	
	\end{align*}
	This can be done by similar argument used for the case $r=(2,0)$ and we omit the details.
\end{proof} \\

\begin{proof}[Proof of Proposition \ref{rates:hitting_times} ]
	We will first sketch a proof for that for any $\epsilon>0$ sufficiently small,
	\[
	\Pi(\sup_{x_0 \in \mathcal{G}} |t^*_{x_0}-t_{x_0}| > \epsilon|\mathbb{D}_n) \xrightarrow{\mathrm{P}_0}0.
	\]
	Recall that $t^*_{\xz}=\text{argmin} \{ |t|\geq 0:  \langle \nabla f^*(\intgx^*(t)), V^*(\intgx^*(t))  \rangle=0, \lambda^*(\intgx^*(t))< 0 \}$. Without loss of generality, we assume $t^*_{x_0}$ is nonnegative. 
	Let $C_{\xz}=\{ t\in[0, t^*_{\xz}-\epsilon]:\langle \nabla f^*(\intgx^*(t)) ,V^*(\intgx^*(t)) \rangle=0    \}$.  Note that it suffices to show the following set of assertions hold with posterior probabilities going to $1$,
\begin{enumerate}[label={(\arabic*)}]
	\item  $\inf_{\xz \in \mathcal{G}, t\in C_{\xz} } \lambda(\intgx(t))>0$,
	\item   $\inf_{\xz \in \mathcal{G}, t\in [0, t^*_{\xz}-\epsilon]\backslash C_{\xz} } |\nabla f(\intgx(t)) ,V(\intgx(t)) \rangle|>0$,
	\item  $\sup_{\xz \in \mathcal{G}, t\in [t^*_{\xz}-\epsilon,t^*_{\xz}+\epsilon] }\lambda(\intgx(t))<0$,
	\item  $\text{For all } \xz \in \mathcal{G}, \text{there exists } \tilde{t}_{\xz} \in [t^*_{\xz}-\epsilon,t^*_{\xz}+\epsilon ] \text{ such that }$
      \[ \nabla f(\intgx(\tilde{t}_{\xz})) ,V(\intgx(\tilde{t}_{\xz})) \rangle=0. \]
\end{enumerate}
When $C_{\xz}$ is empty, we can omit condition (1). By invoking the arguments similar to that in the proof of Proposition A.1 of \citet{qiao2016annals}, with $\mathrm{P}_0$-probability tending to one, the above assertions hold with high posterior probability.
	
	Now we are ready to establish the contraction rate.
	The proof of the bounds needed to establish this one is similar to that of Proposition 5.1 of \citet{qiao2016annals}. Let $D_{\intgx,V}f(t):=\nabla f(\intgx(t))^T V(\intgx(t))$ and
	\begin{align*}
		D^2_{\intgx,V}f(t)&:=\langle  \nabla \langle \nabla f(\intgx(t)), V(\intgx(t))  \rangle, V(\intgx(t)) \rangle \\
		&=\nabla f(\intgx(t))^T \nabla V(\intgx (t)) V(\intgx(t))+ V(\intgx(t))^T Hf(\intgx(t))  V(\intgx(t)).
	\end{align*}
	We have for some $\tilde{t}$ between $t_{\xz}$ and $t^*_{\xz}$,
	\[
	0=D_{\intgx,V}f(t_{\xz})=D_{\intgx,V}f(t^*_{\xz})+D^2_{\intgx,V}f(\tilde{t})(\tilde{t}-t^*_{\xz}).
	\]
	We claim that
	\[
	\Pi (\inf_{\xz \in \mathcal{G}}| D^2_{\intgx,V}f(\tilde{t}) | > \eta | \mathbb{D}_n ) \xrightarrow{\mathrm{P}_0}0.
	\]
	Since
	\[\sup_{x_0 \in \mathcal{G}} \Big|  D^2_{\intgx,V}f(\tilde{t}) -  D^2_{\intgx^*,V^*}f^*(t^*_{\xz})\Big|\geq \Big| \inf_{\xz \in \mathcal{G}}|D^2_{\intgx,V}f(\tilde{t}) |- \inf_{\xz \in \mathcal{G}}| D^2_{\intgx^*,V^*}f^*(t^*_{\xz})|  \Big|
	\]
	and by Assumption (A3), $\inf_{\xz \in \mathcal{G}}| D^2_{\intgx^*,V^*}f^*(t^*_{\xz})| \geq \eta$, it suffices to show that with $\mathrm{P}_0$-probability tending to $1$, $\sup_{x \in \mathcal{G}} \Big|  D^2_{\intgx,V}f(\tilde{t}) -  D^2_{\intgx^*,V^*}f^*(t^*_{\xz})\Big|$ is small with high posterior probability.
	To see this, we write
	\begin{align*}
		\sup_{x_0} \Big|  D^2_{\intgx,V}f(\tilde{t}) -  D^2_{\intgx^*,V^*}f^*(t^*_{\xz})\Big|
		\leq& \sup_{x_0} \Big|  D^2_{\intgx,V}f(\tilde{t}) -  D^2_{\intgx^*,V^*}f^*(\tilde{t})\Big|  \\
		+&  \sup_{x_0} \Big|  D^2_{\intgx^*,V^*}f^*(\tilde{t}) -  D^2_{\intgx^*,V^*}f^*(t^*_{\xz})\Big|.
	\end{align*}
	The first term contracts to zero, simply due to the uniform contraction results for $\nabla f(x)$, $V(x)$, $\nabla V(x)$, $Hf(x)$ (Lemma \ref{posterior_rates_f}) and for $\intgx$ (Theorem \ref{rates:integral_curves}). The second term contracts to zero, since $\intgx^*(t)$ is continuous in $t$ and $\nabla f^*(x)$, $V^*(x)$, $\nabla V^*(x)$, $Hf^*(x)$ are uniform continuous. \\
Also, $\sup_{\xz}| D_{\intgx,V}f(t^*_{\xz})|=\sup_{\xz}|D_{\intgx,V}f(t^*_{\xz})- D_{\intgx^*,V^*}f(t^*_{\xz}) | $ is given by
\begin{align*}
    & \sup_{\xz} \left| \langle \nabla f(\intgx(t^*_{\xz}))- \nabla f^*(\intgx^*(t^*_{\xz}))    , V(\intgx(t^*_{\xz}))-V^*(\intgx^*(t^*_{\xz}))   \rangle \right| \\
	+& \sup_{\xz} \left| \langle  \nabla f(\intgx(t^*_{\xz}))- \nabla f^*(\intgx^*(t^*_{\xz}))  ,V^*(\intgx^*(t^*_{\xz}))   \rangle \right| \\
	+& \sup_{\xz} \left| \langle  \nabla f^*(\intgx^*(t^*_{\xz}))   , V(\intgx(t^*_{\xz}))-V^*(\intgx^*(t^*_{\xz}))  \rangle \right|.
\end{align*}
	Now since
	\[
	\sup_{\xz}|t_{\xz}-t^*_{\xz}| \leq \frac{1}{\inf_{\xz}D^2_{\intgx,V}f(\tilde{t})} \sup_{\xz}|D_{\intgx,V}f(t^*_{\xz})|,
	\]
	we have
	\begin{align*}
		\sup_{x_0 \in \mathcal{G}} |t^*_{x_0}-t_{x_0}|
		\lesssim & \sup_{x_0 \in \mathcal{G}} \left(\| \nabla f (\intgx(t^*_{x_0}))- \nabla f^*(\intgx^*(t^*_{x_0}))\|  \right) \sup_{x_0 \in \mathcal{G}} \|V^*(\intgx^*(t^*_{x_0}))\| \\
		+ & \sup_{x_0 \in \mathcal{G}} \left( \| V(\intgx(t^*_{x_0}))-V^*(\intgx^*(t^*_{x_0})) \| \right) \sup_{x_0 \in \mathcal{G}} \|\nabla f^*(\intgx^*(t^*_{x_0}))\|.
	\end{align*}
	It is easy to see that
	\[
	\sup_{x_0 \in \mathcal{G}}\| \nabla f (\intgx(t^*_{x_0}))- \nabla f^*(\intgx^*(t^*_{x_0}))\| \lesssim \sup_{x_0 \in \mathcal{G}}\| \nabla f(x)-\nabla f^*(x)\|+\sup_{x_0 \in \mathcal{G}}\sup_{t \in [0,T^*_{x_0}]}\|\intgx(t)-\intgx^*(t) \|,
	\]
	\[
	\sup_{x_0 \in \mathcal{G}}\| V (\intgx(t^*_{x_0}))- V^*(\intgx^*(t^*_{x_0}))\| \lesssim \sup_{x_0 \in \mathcal{G}}\| V(x)-V^*(x)\|+\sup_{x_0 \in \mathcal{G}}\sup_{t \in [0,T^*_{x_0}]}\|\intgx(t)-\intgx^*(t) \|.
	\]
	With the choice of $J \asymp (n/\log n)^{1/(1+ 2\alpha)}$, by Lemma \ref{posterior_rates_f},
	$\sup_{x}\| \nabla f(x)-\nabla f^*(x)\| $ has posterior contraction rate of order $(n/ \log n)^{(3-2\alpha)/(2(1+2\alpha))}$, while $ \sup_{x}\| V(x)-V^*(x)\|$ has that of order $ (n/ \log n)^{(5-2\alpha)/(2(1+2\alpha))}$.
	Therefore, considering the rate from Theorem \ref{rates:integral_curves}, we have the desired result.
\end{proof}\\

\begin{proof}[Proof of Theorem \ref{rates:filaments_on_curves} ]
	One can write
	\begin{align*}
		\intgx(t_{x_0})-\intgx^*(t_{x_0}^*)&=\intgx(t_{x_0})-\intgx(t_{x_0}^*)+ \intgx(t_{x_0}^*)-\intgx^*(t_{x_0}^*)\\
        &= V(\intgx(\tilde{t}_{\xz}))(t_{\xz}-t^*_{\xz})+\intgx(t_{x_0}^*)-\intgx^*(t_{x_0}^*)\\
		&=V^*(\intgx^*(t^*_{\xz}))(t_{\xz}-t^*_{\xz})+( V(\intgx(\tilde{t}_{\xz})) -V^*(\intgx^*(t^*_{\xz})))(t_{\xz}-t^*_{\xz})\\
		&\quad + \intgx(t_{x_0}^*)-\intgx^*(t_{x_0}^*)
	\end{align*}
	for $\tilde{t}_{\xz}$ between $t_{\xz}$ and $t^*_{\xz}$. Therefore,
	\begin{align*}
		\sup_{\xz \in \mathcal{G}}\|\intgx(t_{x_0})-\intgx^*(t_{x_0}^*) \| \lesssim & \sup_{\xz \in \mathcal{G}} |t_{\xz}-t^*_{\xz}|\Bigl(1+\sup_{\xz \in \mathcal{G}}\|V(\intgx(\tilde{t}_{\xz})) -V^*(\intgx^*(t^*_{\xz}))\| \Bigr) \\
		&+ \sup_{\xz \in \mathcal{G}}\|\intgx(t_{x_0}^*)-\intgx^*(t_{x_0}^*)\|.
	\end{align*}
	In view of the posterior contraction of $V$ (Lemma \ref{posterior_rates_f}), Theorem \ref{rates:integral_curves} and Proposition \ref{rates:hitting_times}, the posterior of
	$\sup_{\xz \in \mathcal{G}}\|V(\intgx(\tilde{t}_{\xz})) -V^*(\intgx^*(t^*_{\xz}))\|$ contracts to zero.
	
	Since
	\begin{align*}
		\sup_{\xz \in \mathcal{G}}\|\intgx(t_{x_0})-\intgx^*(t_{x_0}^*) \| &\lesssim  \sup_{\xz \in \mathcal{G}}|t_{\xz}-t^*_{\xz}|+ \sup_{\xz \in \mathcal{G}}\|\intgx(t_{x_0}^*)-\intgx^*(t_{x_0}^*)\|,
	\end{align*}
	by Theorem \ref{rates:integral_curves} and Proposition \ref{rates:hitting_times}, we establish the posterior contraction rate for $\sup_{\xz \in \mathcal{G}}\|\intgx(t_{x_0})-\intgx^*(t_{x_0}^*) \|$.
\end{proof}\\

\begin{proof}[Proof of Proposition \ref{twins}]
Consider (A2) for (i). Note that
\begin{align*}
	&\lambda(\intgx(t_{\xz}))-\lambda^*(\intgx^*(t^*_{\xz})) \\
	=&V(\intgx(t_{\xz}))^THf(\intgx(t_{\xz}))V(\intgx(t_{\xz})) -V^*(\intgx^*(t^*_{\xz}))^THf^*(\intgx^*(t^*_{\xz}))V^*(\intgx^*(t^*_{\xz})).
\end{align*}
Thus by the posterior uniform contraction and uniform continuity of $V$, $Hf$ (Lemma \ref{posterior_rates_f}) and the uniform contraction of $\intgx(t_{\xz})$ (Theorem \ref{rates:filaments_on_curves}) over $\xz$, with $\mathrm{P}_0$-probability tending to $1$, clearly condition (A2) holds with high posterior probability. Condition (A3) trivially holds by the same uniform contraction results and noting that
\begin{align*}
	 & \langle  \nabla \langle \nabla f(\intgx(t_{\xz})), V(\intgx(t_{\xz}))  \rangle, V(\intgx(t_{\xz}))  \rangle \\
	=&\nabla f(\intgx(t_{\xz}))^T \nabla V(\intgx(t_{\xz})) V(\intgx(t_{\xz}))+ V(\intgx(t_{\xz}))^T Hf(\intgx(t_{\xz}))  V(\intgx(t_{\xz})).
\end{align*}
The compactness follows trivially by the continuity of $\lambda(x)$ and $\langle \nabla f(x), V(x)  \rangle$ and the fact that intersection of closed sets is closed and the boundedness of $[0,1]^2$. Turning to Condition (A5): fix any $x_0 \in \mathcal{G}$ and any $u, s$ such that $t^*_{x_0}-a^*\leq s<u\leq t^*_{x_0}+a^*$.
\begin{align*}
\frac{1}{u-s}\Big \| \intgx(u)-\intgx(s) \Big\|&= \frac{1}{u-s}\Big \| \big(\intgx(u)- \intgx^*(u)\big) -\big(\intgx(s)-\intgx^*(s)\big)+ \intgx^*(u)-\intgx^*(s) \Big\| \\
&\geq \frac{1}{u-s}\Big \| \intgx^*(u)-\intgx^*(s) \Big\|- \frac{1}{u-s}\Big \| \big(\intgx(u)- \intgx^*(u)\big) -\big(\intgx(s)-\intgx^*(s)\big) \Big\|, \\
&\geq C_{\mathcal{G}}-\epsilon,
\end{align*}
for $\epsilon>0$ arbitrarily small, due to the contraction result in Theorem \ref{rates:integral_curves}.  Finally, in view of Lemma \ref{convergence_rates_f} and Theorem \ref{twins:frequentist}, (ii) holds by similar argument.
\end{proof}\\

\begin{proof}[Proof of Lemma \ref{bounds:hausdorff}  ]
Let $\xz$ be an arbitrary point on $\mathcal{\hat{L}}$. Thus $\intgx(0)=\xz \in \mathcal{\hat{L}}$. Suppose $\intgx(t_{\xz}) \in \mathcal{L}$ for some $t_{\xz}\neq 0$. Such $t_{\xz}$ exists in view of the assumptions of the lemma and the Remark \ref{misc:integral_curves} (1). If $t_{\xz}=0$, then the upper bound holds trivially. Note that $\inf_{y \in \mathcal{L}} \|\xz-y \| \leq \|\intgx(0)-\intgx(t_{\xz})\|\leq |t_{\xz}|$  as $\|V(x)\|=1$.
Let $D_{\intgx,V}f(t):=\frac{d}{dt}f(\intgx(t))=\nabla f(\intgx(t))^T V(\intgx(t))$ and
\begin{align*}
	D^2_{\intgx,V}f(t)&:=\langle  \nabla \langle \nabla f(\intgx(t)), V(\intgx(t))  \rangle, V(\intgx(t)) \rangle \\
	&=\nabla f(\intgx(t))^T \nabla V(\intgx (t)) V(\intgx(t))+ V(\intgx(t))^T Hf(\intgx(t))  V(\intgx(t)),
\end{align*}
where the second equality is due to the chain rule.
A Taylor expansion yields that
\[
D_{\intgx,V}f(t)-D_{\intgx,V}f(t_{\xz})=\bigl(t-t_{\xz}\bigr)D^2_{\intgx,V}f(\tilde{t})
\]
for some $\tilde{t}$ between $0$ and $t_{\xz}$. In particular, since $D_{\intgx,V}f(t_{\xz})=0$, letting $t=0$, we obtain
\[
D_{\intgx,V}f(0)=-t_{\xz}D^2_{\intgx,V}f(\tilde{t}).
\]
Furthermore,
\begin{align*}
	|D_{\intgx,V}f(0)|&=|D_{\intgx,V}f(0)-D_{\intgx,\hat{V}}\hat{f}(0)| \\
	&=|\nabla f(\intgx(0))^T V(\intgx(0))-\nabla \hat{f}(\intgx(0))^T \hat{V}(\intgx(0))| \\
	&\leq \sup_{x}|\nabla f(x)^T V(x)-\nabla \hat{f}(x)^T \hat{V}(x) | \\
	&\leq \sup_{x}|\nabla f(x)^T (V(x)-\hat{V}(x)) |+\sup_{x}|(\nabla f(x)-\nabla \hat{f}(x) )^T\hat{V}(x))| \\
	&\leq c_0 (\sup_{x}\|V(x)-\hat{V}(x) \|+\sup_{x}\| \nabla f(x)-\nabla \hat{f}(x)\|)\\
    & \leq c_0 \sup_{x}\|V(x)-\hat{V}(x) \|,
\end{align*}
where $c_0=(\sup_x\| \nabla f(x)\| )\vee 1$ (recalling $\|V\|=1$).
By the uniform continuity of $\nabla f(x)$, $\nabla V(x)$, $V(x)$ and $Hf(x)$ and the continuity of $\intgx(t)$ in $t$,  without loss of generality, we can make $|t_{\xz}|$ small enough (see comments in the end of the proof). Thus we have
\[
\left|D^2_{\intgx,V}f(\tilde{t})-D^2_{\intgx,V}f(t_{\xz}) \right| <\eta/2.
\]
By Assumption (A3), $\left|D^2_{\intgx,V}f(t_{\xz})\right| >\eta$, and hence
\begin{align*}
	\left|D^2_{\intgx,V}f(\tilde{t})\right|
	&\geq\left|D^2_{\intgx,V}f(t_{\xz})\right| -\left|D^2_{\intgx,V}f(\tilde{t})-D^2_{\intgx,V}f(t_{\xz})\right| >\eta-\eta/2=\eta/2.
\end{align*}
Therefore,
$\inf_{y \in \mathcal{L}} \|\xz-y \| \leq |t_{\xz}| \leq \frac{2c_0}{\eta} \sup_{x}\|V(x)-\hat{V}(x)\|$. Thus $d(\mathcal{\hat{L}}|\mathcal{L}) \leq \frac{2c_0}{\eta} \sup_{x}\|V(x)-\hat{V}(x)\|$. Similarly, $d(\mathcal{\mathcal{L}|\hat{L}}) \leq \frac{2\tilde{c}_0}{\eta} \sup_{x}\|V(x)-\hat{V}(x)\|$ for some fixed constant $\tilde{c}_0>0$. Therefore, $\mathrm{Haus}(\mathcal{L},\mathcal{\hat{L}})\leq \frac{c}{\eta}\sup_{x \in [0,1]^2}\|V(x)-\hat{V}(x)\|$ for some fixed constant $c>0$.  In principle, we can use different $\eta$ in Assumption (A3) for $f$ and $\hat{f}$. For clarity, we use the same $\eta$. Allowing different constants do not change the final conclusion.
Recall that $V(x)=G(df^2(x))$. Now since $G$ is a fixed Lipschitz continuous function, it is easy to get the upper bound for  $\sup_{x \in [0,1]^2}\|V(x)-\hat{V}(x)\|$ in terms of the supremum distance of the derivatives of $f(x)-\hat{f}(x)$. Indeed, since 
$$
\sqrt{(w_1-u_1)^2+4v_1^2}- \sqrt{(w_2-u_2)^2+4v_2^2}\leq |u_1-u_2|+|w_1-w_2|+2|v_1-v_2|,
$$
applying $c_r$-inequality, one can get 
\begin{align*}
 \sup_{x \in [0,1]^2} \|V(x)-\hat{V}(x)\| &\leq \sqrt{60}   \sup_{x \in [0,1]^2}|f^{(2,0)}(x)-\hat{f}^{(2,0)}(x) | \\
&+ \sqrt{60}    \sup_{x \in [0,1]^2}| f^{(1,1)}(x) -\hat{f}^{(1,1)}(x)|\\
&+\sqrt{216}  \sup_{x \in [0,1]^2}|f^{(0,2)}(x)-\hat{f}^{(0,2)}(x)|.
\end{align*}
At last, $|t_{\xz}|$ can be made arbitrarily small. To see this, if Assumption (A5) holds for the $f$ (or more precisely, ${\Upsilon}_{\xz}$), then $\|\intgx(t_{\xz})-\hat{\Upsilon}_{\xz}(\hat{t}_{\xz})\|=\|\intgx(t_{\xz})-\xz\|=\|\intgx(t_{\xz})-
\intgx(0)\|>C_{\mathcal{G}} |t_{\xz}|$. Since $\| \hat{\Upsilon}_{\xz}(\hat{t}_{\xz})-\intgx(t_{\xz})\|$ can be made arbitrarily small due to being close in supremum norm over ${\xz}$ (see previous theorems; for instance, Theorem \ref{rates:filaments_on_curves}, if here $\hat{f}$ is taken as $f$ from posterior samples, $f$ as the truth $f^*$), $|t_{\xz}|$ can be made arbitrarily small. 
This completes the proof.
\end{proof}\\

\begin{proof}[Proof of Theorem \ref{credible_sets}]
For $\gamma<1/2$, by argument in the proof of Theorem of 5.3 of \citet{yoo2016supremum}, one can establish that, for $ r \in \mathcal{R}=\{(2,0),(1,1),(0,2)\}$,
\begin{align*}
	&R_{n,r,\gamma} \asymp {\rm  E}\Big(   \|f^{(r)}- \tilde{f}^{(r)} \|_{\infty} | \mathbb{D}_n \Big), \\
	& R^2_{n,r,\gamma}\asymp (\log n/n)J^6 \asymp (\log n/n )^{(\alpha-2)/(\alpha+1)}.
\end{align*}

Recall from  (\ref{posterior:eqn}), $\Pi(f^{(r)} |\mathbb{D}_n) \sim \text{GP}( \tilde{f}^{(r)}, \hat{\sigma}^2 \Sigma_r)$. By Borell's inequality (see Proposition A.2.1 of \citet{van1996weak}),
\begin{align*}
	\Pi(\mathcal{L}\notin C_{\mathcal{L}}|\mathbb{D}_n)&\leq \Pi(\mathcal{L} \text{ corresponds to } f\notin C_{f,r,\gamma}^{\rho} \text{ for some } r \in\mathcal{R}|\mathbb{D}_n) \\
	&\leq \sum_{r\in \mathcal{R}} \Pi( \|f^{(r)}- \tilde{f}^{(r)} \|_{\infty}>\rho R_{n,r,\gamma}|\mathbb{D}_n)  \\
	&\leq 3\max_{r \in \mathcal{R}}\exp[-c^2R^2_{n,r,\gamma}/c_{n,r}],
\end{align*}
for some constant $c>0$ and $c_{n,r}=\sup_{x}\text{var}(f^{(r)}(x)-\tilde{f}^{(r)}(x)|\mathbb{D}_n)$ which is bounded by a  constant multiple  of
 \[\sup_{x}\Sigma_r(x,x) \lesssim \sup_{x}\|b^{(r)}_{J_1,J_2}(x) \|^2 \|(B^T B+ \Lambda^{-1}_0)^{-1}\|_{(2,2)} \lesssim (\log n)^{-3/(\alpha+1)}n^{(2-\alpha)/(\alpha+1)}.
 \]
 Therefore, the above posterior probability tends to zero. Finally, by Lemma \ref{convergence_rates_f} and noting that $R_{n,r,\gamma}$ is at least as big as $\delta_{n,2,J}$ with high probability, $\mathrm{P}_0(\mathcal{L}^* \in C_{\mathcal{L}})=\mathrm{P}_0( \|f^{*(r)}- \tilde{f}^{(r)} \|_{\infty}\leq \rho R_{n,r,\gamma}, \forall r \in \mathcal{R} ) \to 1$,  establishing the coverage of $C_{\mathcal{L}}$.

To see $C_{\mathcal{L}} \subset \bar{C}_{\mathcal{L}} $, for any $\mathcal{L} \in C_{\mathcal{L}} $, it is induced by some $f$ such that $ \|f^{(r)}-\tilde{f}^{(r)}\|_{\infty}\leq \rho \max_{r \in \mathcal{R} } R_{n,r,\gamma}$. In view of the discussion in the beginning of this section and Proposition \ref{twins}, since Lemma \ref{bounds:hausdorff} holds with $\mathrm{P}_0$-probability tending to $1$ with $\mathcal{L}$ being the filament in posterior and $\hat{\mathcal{L}}$ being the filament induced by the posterior mean $\tilde{f}$,  the result immediately follows.
\end{proof}

\bibliographystyle{dcu}
\bibliography{filament_24March2020}

\end{document}